\documentclass[11pt]{amsart}
\usepackage{amsmath}
\usepackage{amssymb}
\usepackage{tabularx}
\usepackage{enumerate}
\usepackage{graphicx}
\usepackage{texdraw}
\usepackage{dsfont}
\usepackage{epstopdf}
\usepackage{hyperref}
\usepackage{color}

\usepackage{mathrsfs}
\usepackage{amsfonts,amssymb,amsmath}
\usepackage{epsfig}

\makeatletter
\@addtoreset{equation}{section}

\makeatother
\newtheorem{thm}{Theorem}[section]
\newtheorem{lem}[thm]{Lemma}

\newtheorem{remark}[thm]{Remark}

\newcommand\eqnref[1]{{\rm(\ref{#1})}}

\newcommand{\R}{\mathbb{R}}
\newcommand{\ve}{\varepsilon}

\newcommand\cA{{\mathcal A}}
\newcommand\cB{{\mathcal B}}
\newcommand\bes{\begin{eqnarray}}
\newcommand\ees{\end{eqnarray}}
\newcommand\bess{\begin{eqnarray*}}
\newcommand\eess{\end{eqnarray*}}
\newcommand\bE{{\mathbf E}}
\newcommand\BbbR{{\mathbb R}}
\newcommand\cL{{\mathcal L}}
\newcommand\cM{{\mathcal M}}
\newcommand\vep{{\varepsilon}}
\newcommand\mscX{{\mathscr{X}}}
\newcommand\Shift{{\mathbf{Z}}}

\begin{document}

\title[Long time behavior of solutions]
{Long time behavior of solutions of a reaction-diffusion equation on unbounded intervals  with
Robin boundary conditions}
\thanks{This work is  supported in part by NSFC (No. 11271285) and by NSF-1008905. The last author is supported by JSPS}
\author[]{}
\author[X. Chen, B. Lou, M. Zhou, T.Giletti]{Xinfu Chen$^1$, Bendong Lou$^2$, Maolin Zhou$^3$, Thomas Giletti$^{3,4}$}
\thanks{$^1$  Department of Mathematics,   University of Pittsburgh,
Pittsburgh, PA 15260, USA}
\thanks{$^2$ Department of Mathematics, Tongji University, Shanghai 200092, China.}
\thanks{$^3$ Graduate School of Mathematical Sciences, The University of Tokyo,
Tokyo 153-8914, Japan}
\thanks{$^4$ Institut Elie Cartan de Lorraine, Universit\'{e} de Lorraine, Vandoeuvre-l\`{e}s-Nancy 54506, France.}
\thanks{{\bf Emails:} {\sf xinfu@pitt.edu} (X. Chen), {\sf
blou@tongji.edu.cn} (B. Lou), {\sf zhouml@ms.u-tokyo.ac.jp} (M. Zhou),
{\sf thomas.giletti@univ-lorraine.fr} (T. Giletti) }
\begin{abstract}
We study the long time behavior, as $t\to\infty$,  of solutions of
$$
\left\{
\begin{array}{ll}
u_t = u_{xx} + f(u), &  x>0, \ t >0,\\
u(0,t) = b  u_x(0,t), & t>0,\\
u(x,0) = u_0 (x)\geqslant 0 , & x\geqslant    0,
\end{array}
\right.
$$
where $b\geqslant 0$ and $f$  is an unbalanced bistable nonlinearity. By investigating families of initial data of the type $\{ \sigma \phi \}_{\sigma >0}$, where $\phi$ belongs to an appropriate class of nonnegative compactly supported functions, we exhibit the sharp threshold between vanishing and spreading. More specifically, there exists some value $\sigma^*$ such that the solution converges uniformly to 0 for any $0 < \sigma < \sigma^*$, and locally uniformly to a positive stationary state for any $ \sigma > \sigma^*$. In the threshold case $\sigma= \sigma^*$, the profile of the solution approaches the symmetrically decreasing ground state with some shift, which may be either finite or infinite. In the latter case, the shift evolves as $C \ln t$ where~$C$ is a positive constant we compute explicitly, so that the solution is traveling with a pulse-like shape albeit with an asymptotically zero speed. Depending on $b$, but also in some cases on the choice of the initial datum, we prove that one or both of the situations may happen.
\end{abstract}

\subjclass[2010]{35K57, 35K15, 35B40}
\keywords{Reaction-diffusion equation, long time behavior, Robin boundary condition, sharp threshold.}
\maketitle


\section{Introduction}

In this paper, we investigate the long time behavior, as $t\to \infty$, of solutions of
\begin{equation}\label{p}
\left\{
\begin{array}{ll}
u_t = u_{xx} + f(u), &  x>0, \ t >0,\medskip\\
u(0,t) = b  u_x(0,t), & t>0,\medskip\\
u(x,0) = u_0 (x) \geqslant 0, & x\geqslant    0,
\end{array}
\right.
\end{equation}
where $b\geqslant 0$ is a constant and $f$ is an {\bf unbalanced  bistable} nonlinearity satisfying
\begin{equation*}\leqno{\bf (F)}
\hskip 10mm
\left\{
 \begin{array}{l}
 f\in C^1([0,\infty)), \  f(0)=0> f'(0) =: -\lambda^2, \
 \ f(\cdot) <0 \mbox{ in } (0,\alpha),\medskip \\  f(\cdot)>0 \mbox{ in } (\alpha,1),
 \  f(\cdot)<0 \mbox{ in }(1,\infty),\  \inf_{s>1} f'(s)>-\infty,\medskip
  \\
\mbox{for } F(u):= -2\int_0^u f(s) ds,\ F(\theta)=0 \mbox{ for some } \theta \in (\alpha,1).
\end{array}
\right.
\end{equation*}
Such nonlinearities appear in various applications including
mathematical ecology, population genetics and physics. An interesting feature is that the outcome depends critically on the initial datum (see the seminal papers~\cite{AW1,AW2} for the Cauchy problem in the whole space). Here, the initial function~$u_0$ belongs to  $\mathscr {X}(h)$ for some $h>0$, where
\bess
\mathscr {X}(h) := \{ \phi  \mid
 \phi \in C ([0,\infty)),\ \phi\geqslant 0,\ \phi\not \equiv 0 \mbox{ and }\phi \equiv 0 \mbox{ in }[h,\infty)\}.
\eess

It easily follows from the comparison principle that solutions associated with such initial data remain positive and are uniformly bounded with respect to both space and time. Therefore, one can expect the large time behavior of solutions to be largely dictated by nonnegative and bounded steady states of \eqref{p}, that is by solutions of
\begin{equation}
 \label{stationary} v''+f(v)=0\leqslant v\hbox{ \ in \ }[0,\infty), \quad  v(0)=b v'(0), \quad v\in L^\infty((0,\infty)).
\end{equation}
A  phase plane analysis shows that all such steady states  can be classified as follows (c.f. \S2):

\begin{itemize}
\item[(1)]  {\bf Trivial Solution}  $v\equiv 0$;

\item[(2)]  {\bf Active States} $v (\cdot) = v_* (\cdot -z)$ where $v_*$ is the unique increasing solution of $v_* '' + f(v_*) =0$ on $[0,\infty)$ subject to $v_* (0)=0$, $v_* (\infty)=1$, and $$z \in \Shift_{active} (b) := \{ z \; | \; v_* (-z) = b v_* ' (-z) \} \neq \emptyset ;$$

\item[(3)]  {\bf Ground States}  $v(\cdot)=V(\cdot-z)$ where  $V$ is the unique even positive solution of $V''+f(V)=0$ on $\R$ subject to $V(\infty)=0$, and $$z\in \Shift_{ground} (b):=\{z\;|\; V(-z)=b V'(-z)\};$$

\item[(4)] {\bf Positive Periodic Solutions}.
\end{itemize}
We will see in Section~\ref{subsec:stationary sol}, using standard phase plane analysis, that the set $\Shift_{ground} (b)$ can be characterized as follows:
\begin{equation}\label{ground_set}
\Shift_{ground} (b) = \{ z > 0 \; | \; V(-z)=s \mbox{ and } b \sqrt{F(s)} = s \}.
\end{equation}
In particular, ground states of \eqref{p} exist or, in other words, $\Shift_{ground} (b)$ is not empty, if and only if there exists $s \in (0,\theta)$ such that $b \sqrt{F(s)}= s$.

Note that all ground states of \eqref{p} are shifts of the same function $V$. This function $V$ always exists and is itself often refered to as the ground state of the associated Cauchy problem on the whole real line (see \eqref{p-Cauchy} below). Therefore, by some slight abuse of language and for convenience, we will often refer to any function $V(\cdot -\xi)$, with $\xi \in \R$, as a shifted ground state, even though it may not satisfy the Robin boundary condition.\\

Our first main result is  the  following:


\begin{thm}\label{thm:main}
Assume {\bf (F)} and $\phi\in \mathscr{X}(h)$ for some $h>0$. Let  $u$ be the
solution of  \eqref{p} with  $u_0 =\sigma \phi$ ($\sigma \geqslant 0$).
There exists $\sigma^*  \in (0,\infty)$ such that the following trichotomy  holds:

\begin{itemize}
\item[\rm (i)] If $\sigma>\sigma^*$, {\bf spreading} happens in the following sense:
\bess \lim\limits_{t\to \infty} \|u(\cdot,t) - v_* (\cdot -z_\sigma)\|_{C^2([0,M])} = 0 \quad \mbox{ for any } M>0,
\eess
where $z_\sigma \in \Shift_{active} (b)$ is a nonincreasing function of $\sigma$.\vspace{3pt}

\item[\rm (ii)] If $0\leqslant \sigma < \sigma^*$, {\bf vanishing} happens in the following sense:
\bess \lim\limits_{t\to \infty}\| u(\cdot,t)\|_{H^2([0,\infty))}  = 0 ;
\eess

\item[\rm (iii)]
In the {\bf transition} case $\sigma =\sigma^*$, the solution converges to a shifted ground state:
\bess \lim\limits_{t\to \infty} \|u(\cdot,t) - V(\cdot -\xi (t))\|_{H^2([0,\infty))}=0,
\eess
where either $\lim\limits_{t\to \infty}\xi (t)=\infty$ or $\lim\limits_{t\to \infty} \xi (t)=z\in \Shift_{ground}(b)$.
\end{itemize}
Moreover, $\sigma^*$ is nonincreasing and continuous with respect to both $b \in [0,\infty)$ and $\phi \in \cup_{h >0}\mathscr {X}(h)$.
\end{thm}

\begin{remark}\rm
This theorem, as well as the ones below, could be extended to some other families of initial data $\{ \phi_\sigma\}_{\sigma >0}$ that are increasing with respect to the parameter~$\sigma$, and such that any element is a bounded, nonnegative and nontrivial compactly supported function. For instance, our results also hold for the family of characteristic functions $u_0(x):=\mathbf{1}_{[0,\sigma]}$, as in~\cite{Zla} where the equation on the whole real line was considered. As the proof is very similar and for simplicity, we choose to restrict ourselves to families of the type $\sigma \phi$ with $\phi \in \cup_{h>0} \mscX (h)$.
\end{remark}

In the next theorem, we will clear out what happens in case~(iii).
We denote the set of all the initial data that fall into the transition case by
$$\Sigma :=\{\phi \mid \phi \in \cup_{h>0} \mscX(h)\ \mbox{  and }\lim\limits_{t\to \infty} \|u(\cdot,t;\phi) - V(\cdot - \xi (t))\|_{H^2([0,\infty))}=0\},$$
and also introduce its subset
$$\Sigma _1:=\{\phi \mid \phi \in \Sigma \mbox{ and }\lim\limits_{t\to \infty} \xi (t)=z\in \Shift_{ground} (b)\}.$$

\begin{thm}\label{thm:main2}
Under the assumptions of Theorem~\ref{thm:main}, when $\sigma =\sigma^*$, we have
$\xi (t)=o(t)$ and even, without loss of generality, that $\xi ' (t) \to 0$ as $t \to +\infty$ and
\begin{itemize}
 \item[{\rm(i)}]  if  $b<\frac{s}{\sqrt{F(s)}}$ for all $s\in (0,\theta )$, then
$$\lim\limits_{t\to \infty}\xi (t)=\infty ;$$
 \item[{\rm(ii)}] if there exists a sequence $s_n\rightarrow 0$ such that $b\geqslant \frac{s_n}{\sqrt{F(s_n)}}$ for all $n\in \mathbb{N}$, then there is a $z \in \Shift_{ground} (b)$ such that
$$\lim\limits_{t\to \infty} \xi (t)=z;      $$

\item[{\rm(iii)}]if none of the two conditions above hold, then both cases will happen depending on $\phi$. Moreover, $\Sigma _1$ is a closed set of $\Sigma $ in $L^{\infty}$-topology.
\end{itemize}
\end{thm}
\begin{remark}\rm
We emphasize that these cases are mutually exclusive, and that Theorem~\ref{thm:main2} covers all possible choices of $f$ satisfying {\bf (F)} and $b \geqslant 0$.
For instance, $\lim\limits_{t\to \infty} \xi (t)=\infty $ happens for some appropriate initial datum when $b\lambda <1$, while $\lim\limits_{t\to \infty}\xi (t)=z\in \Shift_{ground}(b)$ always happen when $b\lambda >1$.
\end{remark}
\begin{remark}\label{zground}\rm
Note that for $\xi (t) \to \infty$ to occur in the transition case, the set $\Shift_{ground} (b)$ needs to be either empty (case (i)) or bounded (case (iii)), thanks to the characterization \eqref{ground_set} of $\Shift_{ground} (b)$. Moreover, whenever $\Shift_{ground} (b)$ is empty, it already followed from Theorem 1.1 that $\xi (t) \to \infty$. However, the boundedness of $\Shift_{ground}(b)$ does not always allow for $\xi (t) \to \infty$, as this situation may fall into either cases (ii) or (iii) depending on whether $b > \frac{s}{\sqrt{F(s)}}$ or $b< \frac{s}{\sqrt{F(s)}}$ for small values of $s$.
\end{remark}
In the transition case and when $\lim_{t\to\infty} \xi (t)=\infty$, the solution slowly drifts away to the right, and it is an interesting problem to study this motion. The next theorem gives a precise calculation result of the position $\xi (t)$, under some technical additional regularity assumptions on the nonlinearity $f$.

\begin{thm}\label{thm:main3}
Under the assumptions of Theorem \ref{thm:main}, when $\sigma = \sigma^*$ and $\lim\limits_{t\to \infty} \xi (t)=\infty$:
\begin{itemize}
\item[\rm (i)]if $b\lambda <1$ and $f\in C^2$, then
 \bess  \xi (t) = \frac{\ln t}{2\lambda} + \frac{\ln [2\lambda c(b)]}{2\lambda} + o(1) \quad \mbox{ as } t\to\infty,
\eess
where $A:= \theta e^{\int_0^\theta[\frac{\lambda}{\sqrt{F(s)}}-\frac{1}{ s}]ds}$ and $c(b) :=\frac{\lambda^2 (1-b\lambda) A^2}{[1+b\lambda] \int_0^\theta \sqrt{F(s)} ds}$;

\item[\rm (ii)]if $b\lambda =1$, $f\in C^3$ and $f''(0)>0$, then
\bess  \xi (t) = \frac{\ln t}{3\lambda} + \frac{\ln[3\lambda \hat{c}]}{3\lambda} + o(1) \quad
\mbox{ as } t\to \infty, \eess
where $\hat c :=\frac{f''(0) A^3}{12\int_0^\theta\sqrt{F(s)}ds}$.
\end{itemize}
\end{thm}

Note that if $b \lambda >1$, by Theorem~\ref{thm:main2}, convergence to a shifted ground state cannot take place with an infinite shift. We also point out that the first order term only depends on $\lambda= \sqrt{-f'(0)}$. In particular, although the value of $b$ is important to determine whether the convergence to a ground state occurs with an infinite shift, the properties of the motion ultimately do not depend so much on the boundary condition.

\begin{remark}\rm
Assertions in Theorem \ref{thm:main3} (ii) extend to the case when $f$ satisfies, for some integer $k\geqslant 2$,
\begin{equation*}\leqno{\bf(Fk)}
{\ } \qquad
 f\in C^{k+1}([0,\infty)),\ f^{(j)}(0)=0\mbox{ for } j=2,3,\cdots,k-1,\\
 \mbox{\ and } f^{(k)}(0)  >0.
\end{equation*}
The analogous  conclusion is that
\bess
\xi (t)= \frac{\ln (t)}{(k+1)\lambda} + \frac{\ln[(k+1)\lambda c_k]}{(k+1)\lambda} + o(1) \quad
\mbox{ as } t\to \infty, \eess
where
$$
c_k : =\frac{f^{(k)}(0) A^{k+1}}{2(k+1)!\int_0^\theta\sqrt{F(s)}ds}.
$$
\end{remark}


 When $b=\infty$, i.e., the Neumann boundary condition
 $u_x(0,t)=0$,  by even reflection, the problem is equivalent to the Cauchy problem
\begin{equation}\label{p-Cauchy}
\left\{
\begin{array}{ll}
 u_t = u_{xx} + f(u), &  x\in \R,\ t>0,\\
 u(x,0) = u_0 (x),& x\in \R
\end{array}
\right.
\end{equation}
with even initial data.
This Cauchy problem has been extensively studied. The  classical papers of Aronson and
 Weinberger \cite{AW1, AW2}
contain  systematic investigation of problem \eqref{p-Cauchy}, with
various sufficient conditions for {\bf spreading} (also known as {\bf propagation})
and {\bf vanishing} (also known as {\bf extinction}).
For other related works, see
\cite{CP,Chen1,DuLou, DM, FM} and the references therein.

Our present work is motivated by Zlato\v{s} \cite{Zla}, Du and Matano \cite{DM} and
Fa\u{s}angov\'{a} and Feireisl \cite{FF}. In \cite{DM}, motivated by a fundamental  result of
Zlato\v{s} \cite{Zla}, a complete description of the asymptotic behavior of the
solution of  \eqref{p-Cauchy} was given (see also Chen \cite{Chen1} for the case of a bistable nonlinearity). More precisely, the authors
first proved that any bounded solution of~\eqref{p-Cauchy} converges to
a stationary one, that is, a solution of
\bess v_{xx} + f(v)=0 \quad\forall\,  x\in \R.
\eess 
When $f$ is of bistable or combustion type, they established  a sharp transition
result: for any nontrivial $\phi \geqslant 0$ with compact support, there exists a sharp
threshold value $\sigma^* >0$ such that spreading happens for $u(\cdot, t;\sigma \phi)$
(the solution of  \eqref{p-Cauchy} with initial data $u_0 = \sigma \phi$) when $\sigma >\sigma^*$,
vanishing happens for $u(\cdot,t;\sigma \phi)$ when $\sigma<\sigma^*$, while
for the threshold value $\sigma^*$, $u(\cdot,t;\sigma^*\phi)$ converges to a
 ground state in the unbalanced  bistable case, or to the ignition point in the combustion case.
  Our theorems extend these sharp transition results
from \eqref{p-Cauchy} to problem \eqref{p}. In our framework, new difficulties arise from the fact that in the transition case, the limiting ground state may be drifting far away from the reference frame, so that a new argument than in \cite{DM, Zla} will be needed.

In \cite{FF}, the authors also studied problem \eqref{p} with Dirichlet boundary condition ($b=0$)
and with bistable $f$ as in {\bf (F)}. They proved that,
for any nonnegative function $\phi \in H^1_0 ([0,\infty))$ which is increasing in $[0,x_0]
(x_0 >0)$ and decreasing in $[x_0, \infty)$, there exist $\sigma_* := \sigma_*(\phi)$ and
$\sigma^* := \sigma^*(\phi)$ with $0<\sigma_* \leqslant \sigma^*$
such that vanishing happens when $\sigma < \sigma_*$, and spreading happens when $\sigma>\sigma^*$. In the transition case $\sigma \in [\sigma_* , \sigma^*]$, the solution converges to a ground state with infinite shift, that is there exists a function $\xi (\cdot)$ satisfying $\lim_{t\to\infty} \xi (t)=\infty$ such that
$\|u(\cdot,t;\sigma \phi)-V(\cdot-\xi (t))\|_{H^1([0,\infty))} \to 0$ as $t\to\infty$. This is quite different from the
problem with Neumann boundary condition  where, in the transition case, $u$ converges to a finitely shifted ground state (c.f. \cite{DM, FF}). Note that \cite{FF} left two important  open problems:
\begin{itemize}
\item[\rm 1.] Whether $\sigma_* =\sigma^*$ or not ?
\item[\rm 2.] How fast does $\xi (t)\to \infty$ as $t\to \infty$ ?
\end{itemize}

This paper is devoted to the connection between the problems with Neumann and Dirichlet boundary conditions, by dealing with the whole range of Robin boundary conditions ($b \geqslant 0$). Theorem~\ref{thm:main} shows that, whatever the boundary condition is, the threshold is always sharp. In terms of applications, this means that the only reasonable outcomes are spreading and vanishing. This is in accordance with prior results in the Neumann boundary condition case. In particular, when $b=0$, Theorem \ref{thm:main} answers positively to the first question above. We point out that our approach even deals with initial data which admit more than a single local maximum point, which were not considered in~\cite{FF}.

However, we already know that the long time behavior of the solution in the transition case does depend on the boundary condition~\cite{FF}. Indeed, although we establish that the solution always converges to a ground state, this convergence may occur with either a finite of infinite shift. A natural question is:
\begin{itemize}
\item[\rm 3.] In the transition case, does the solution approach a finitely shited ground state
$V(\cdot-z)\ (z\in\Shift_{ground} (b))$ or an infinitely shifted ground state $V(\cdot-\xi (t))$ with $\lim_{t\to\infty}\xi (t)=\infty$?
\end{itemize}
Theorem~\ref{thm:main2} answers this question by providing explicit and complete criteria for both cases to occur. In summary, there exists a partition of $[0,\infty]$ into three intervals $I_1 \ni 0$, $I_2$ and $I_3 \ni \infty$  such that:
\begin{itemize}
\item if $b \in I_1$, the shift is always infinite;
\item if $b \in I_3$, the shift is always finite;
\item if $b \in I_2$, both may happen.
\end{itemize}
Whether those intervals are closed or open depends on the shape of $f$, and can be easily determined from
Theorem~\ref{thm:main2}. For instance, when $f$ satisfies $\sqrt{F(s)} < \lambda s$
(a function with the form $s(s-\alpha) (1-s)$ for $s\in [0,1]$ is a typical
example), $I_1 =[0,1/\lambda]$, $I_2 = \emptyset$ and $I_3 = (1/ \lambda , \infty]$.

Lastly, when convergence occurs with an infinite shift, Theorem~\ref{thm:main3} provides an approximation of $\xi(t)$ as $t \to +\infty$ up to the order~$o(1)$, using slow motion on ``center manifold" technique, developed by Carr and Pego \cite{CP}, Fusco and Hale \cite{FuscoHale}, Alikakos, Bates and Fusco~\cite{ABF}, Alikakos and Fusco~\cite{AF}, and Chen et al. \cite{Chen1,Chen2}. In particular, we observe that the motion is asymptotically slow, and that the shift grows logarithmically with respect to time. As the Dirichlet boundary problem is a particular case of \eqref{p} with $b=0$, this also answers the second question raised in~\cite{FF}. It is interesting to note that, as we pointed out before, while the value of $b$ determines whether the motion takes place or not in the transition case, it has little influence on the motion itself.\\

This paper is arranged as follows. In \S2 we give some preliminaries, including a classification of the steady states and an introduction to the zero number argument, which plays an important role in our proofs. In \S3 we give a first convergence result to an a priori not known steady state, and study spreading and vanishing. The last three sections all deal with the transition case. In~\S4, we prove convergence to a shifted ground state, and infer that the threshold is sharp. In \S5 we investigate whether the shift is finite or not. In \S6, we consider the former case and estimate $\xi(t)$ as $t \to +\infty$.


\section{Preliminaries}

\subsection{Decay Rate at $x=\infty$}

\begin{lem}\label{lm:u to 0} Assume {\bf (F)} and $u_0\in \mathscr{X}(h)$ with $h>0$.
Then the solution $u$ of \eqref{p} satisfies
\begin{equation}\label{uest}
u(x,t),\ |u_x(x, t)|,\ |u_{xx}(x,t)|,\ |u_t(x,t)| <
C(t) e^{-\frac{x^2}{16t}} \quad \forall\, t\geqslant0,\  x >2 h.
\end{equation}
\end{lem}

\begin{proof}  Note first that, since $b \geqslant 0$ and $f (s) \leqslant 0$ for all $s \geqslant 1$, it is clear that $\max\{1,\|u_0 \|_\infty\}$ is an upper solution of~\eqref{p}, hence $u (x,t) \leqslant \max \{ 1, \| u_0 \|_{L^\infty} \}$ for all $x \geqslant 0$ and $t \geq0$. This in particular proves, as mentioned before, that the solution~$u$ is uniformly bounded.

Now set $K_1:=\max_{0\leqslant s\leqslant \max\{1,\|u_0\|_\infty\}} f'(s)$ and let $w$ be the solution of
$$
\left\{
 \begin{array}{ll}
   w_t = w_{xx} + K_1w \quad&\forall\,  x\in \R,\ t>0,\\
   w(x,0) = {u}_0 (|x|) &\forall\, x\in \R.
 \end{array}
 \right.
$$
Since $b \geqslant 0$ and by the comparison principle, for any $t>0$ and $x> h$,
\begin{equation}\label{decay-1}
0< u(x,t)  \leqslant   w (x,t) =  \int_{-h}^{h}
u_0 (|y|)\frac{ e^{K_1t-\frac{(x-y)^2}{4 t}}} {\sqrt{4\pi t}}\, dy
\leqslant  \frac{e^{K_1t- \frac{(x-h)^2}{4t} } h\|u_0\|_{\infty} }{\sqrt{\pi t}} .
\end{equation}
The estimate in \eqref{uest} for $u$ follows immediately.
Other conclusions follow from the interior Schauder estimates (c.f. \cite{FM}).
This proves the lemma.
\end{proof}

\subsection{Steady States}\label{subsec:stationary sol} Each solution of
 $v''+f(v)=0$ corresponds to a trajectory $v'{}^2=F(v)-q$ in the $v$-$v'$ phase plane where $q$ is a constant and
 $F(u):=-2\int_0^u f(s)ds$. Furthermore, the boundary condition $v(0)=b v'(0)$ is satisfied whenever the corresponding trajectory intersects the line $v=bv'$. Note that such an intersection may not be unique, so that several steady states of \eqref{p} can be derived from the same trajectory.

As explained before, we are only interested in bounded and nonnegative steady states, that is solutions of~\eqref{stationary}, which thanks to the argument above can be listed as in the lemma below. We refer to the phase plane in Figure \ref{Fig1.1} and omit the details of the proof.

\begin{figure}[tbp]
\begin{center}
\includegraphics[width=2.8in,height=2.8in]{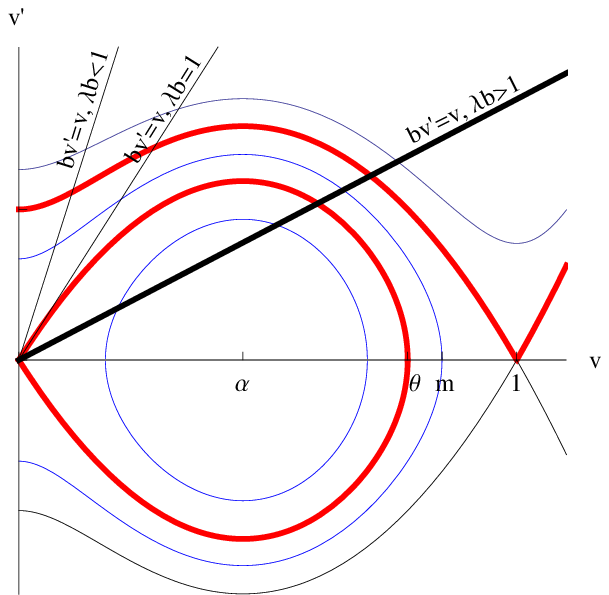}\quad
\includegraphics[width=2.8in,height=2.8in]{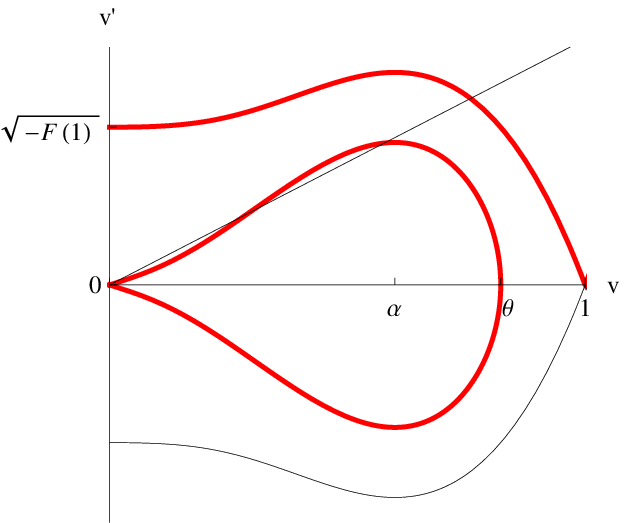}
\end{center}
\caption{\small Trajectories $v'{}^2=F(v)-q$ on phase plane for  $v''+f(v)=0$. \newline  Left: $\frac{\sqrt{F(s)}}{s}$ is strictly monotonic on $[0,\alpha]$; Right: $\frac{\sqrt{F(s)}}{s}$ is not monotonic.}
\label{Fig1.1}
\end{figure}

\begin{lem}\label{limitsol}  Assume {\bf (F)} and $b\geqslant 0$. Then all solutions of the stationary problem \eqref{stationary} are of one of the following types:
\begin{enumerate}
\item[{\rm(1)}]  {\bf Trivial Solution:}  $v\equiv 0$.
\item[{\rm(2)}] {\bf Active States:} $v (\cdot) \equiv v_* (\cdot - z)$,  where $v_*$ is the unique solution of
    \bess
    v_*''+f(v_*)=0<v_*' \hbox{ \ and \ } v_*'^2 = F(v_*) - F(1) \hbox{ \ in \ } [0,\infty), \quad v_*(0)=0, \quad v_*(\infty)=1,\eess
and
$$z \in \Shift_{active} (b) := \{ z < 0 \; | \ v_* (-z)= b v_* ' (-z) \} \neq \emptyset .$$
\item[{\rm(3)}] {\bf Ground States:} $v(\cdot) \equiv V (\cdot - z)$ where $V$ is the unique (symmetrically decreasing) solution of
\bess  V''+f(V)=0<V \hbox{ \ and \ }  V'^2 = F(V) \hbox{ \ in  \ } \R , \quad V (0)=\theta, \quad V(\pm \infty)=0,\eess
and
$$z \in \Shift_{ground} (b) := \{ z \in \R \, | \ V (-z) = b V' (-z) \} \subset (0,\infty).$$
The set $\Shift_{ground} (b)$ is not empty if and only if there exists $s_0 \in (0,\theta)$ such that $b = \frac{s_0}{\sqrt{F (s_0)} }$.

    \item[{\rm(4)}] {\bf Positive Periodic Solutions:} $v>0$ oscillates around $\alpha$ and satisfies $v'{}^2=F(v)-F(m)$ with $m\in[\alpha,\theta)$. These solutions exist when $b$ is large.
\end{enumerate}
\end{lem}

As can be seen from our main results, only the first three types will appear in the long time behavior of solutions of the evolution problem.

In our proofs, we will also need a few other properties. First, we remark that $V$ is
  given implicitly  by
    \bess  |z|=\int^\theta_{V(z)} \frac{ ds}{\sqrt{F(s)}}=-\frac{1}{\lambda} \ln \frac{V(z)}{\theta} + \int^\theta_{V(z)}\Big(\frac{1}{\sqrt{F(s)}}-\frac{1}{\lambda s} \Big){ds},\eess
    where $\lambda=\sqrt{-f'(0)}$.
The asymptotic behavior of $V$ easily follows:
    \bes V(\pm z)= Ae^{-\lambda z} +O(e^{-2\lambda z}),\qquad
    V'(\pm z)=\mp A\lambda e^{-\lambda z} +O(e^{-2\lambda z})\hbox{ \ as \ }z\to\infty,\label{asmpV}\ees
where $$A :=\displaystyle \theta e^{\int_0^\theta [\frac{\lambda}{\sqrt{F(s)}}-\frac{1}{ s}]ds}.$$
When {\bf (Fk)} holds, more precise formulas can be obtained:
  \bess
\begin{array}{l}
V(\pm z)= Ae^{-\lambda z} -H_k e^{-k\lambda z}+ O(e^{-(k+1)\lambda z}) \vspace{5pt}\\
    V'(\pm z)=\mp A\lambda e^{-\lambda z} \pm  k\lambda H_k e^{-k\lambda z} +O(e^{-(k+1)\lambda z})
\end{array}
 \hbox{ \ as \ }z\to\infty ,
\eess
   where
   \bess H_k := \frac{f^{(k)}(0) A^k}{\lambda (k+1)! (k-1)}. \eess

We conclude this section by noting that, for each $m\in(\theta,1)$, the trajectory passing through the point $(v=m,v'=0)$ in the phase plane gives a function $v_m$ satisfying
\bes\label{finite}
\ \ \ v_m''+f(v_m)=0<v_m \leqslant m\hbox{ \ in  \ }(0,2L_m), \quad v_m (0)=v_m (2L_m)=0,
\ees
where  \bes\label{Lm} L_m:=\int_0^m\frac{ds}{\sqrt{F(s)-F(m)}} \in (0,\infty).  \ees
Although this cannot be extended as a nonnegative steady state of~\eqref{stationary}
(note that $v_m ' (0) > 0 > v'_m (2L_m)$), it can be extended continuously on the whole real line by letting
\bes\label{extended_vm}
 v_m\equiv 0 \hbox{ \ on \ } (-\infty,0]\cup[2L_m,\infty).
\ees
This is a standard procedure to construct generalized lower solutions for the reaction-diffusion equation,
which will also prove useful in our framework.

\subsection{Zero Number of $u_x(\cdot,t)$}
We use $\mathcal{Z}(w)$ to denote the number of sign changes of a
continuous function $w(x)$ defined on the whole domain $ [0,\infty)$. Note that, if the zeros
of $w$ are all simple, then $\mathcal{Z}(w)$ coincides with the number of zeros of
$w$ in $(0,\infty)$. We also use the notation $\mathcal{Z}_I (w)$ to denote the number
of sign changes of $w$ on an interval $I\subseteq [0,\infty)$.

\begin{lem}[Lemma 2.3 in \cite{DM}]\label{lem:zero}
Let $w(x,t) \not\equiv 0$ be a solution of  the equation
\bes\label{eq-w}
w_t=w_{xx}+c(x,t)w  \quad\forall\, t\in(t_1,t_2), \; x\in I \subseteq [0,\infty),
\ees
where $c(x,t)$ is bounded. Suppose that, on the boundary of $I$, either $w \equiv 0$ or $w$ never vanishes. Then, for each $t\in(t_1,t_2)$, the zeros of the function $w(\cdot,t)$ do not
accumulate in $\bar{I}$. Furthermore,
\begin{itemize}

\item[{\rm(1)}] $\mathcal{Z}_I(w(\cdot,t))$ is nonincreasing in $t$;

\item[{\rm(2)}] if $w(x^*,t^*)=w_x(x^*,t^*)=0$ for some $t^* \in (t_1,t_2)$ and $x^*\in \bar{I}$, then
\[
\mathcal{Z}_I(w(\cdot,t)) > \mathcal{Z}_I(w(\cdot,s)) \mbox{ for all } t\in(t_1,t^*),\ s\in(t^*,t_2)
\]
whenever $\mathcal{Z}_I (w(\cdot,s))<\infty$.
\end{itemize}
\end{lem}

\begin{lem}\label{lem:intersection}
Let $u$ be a solution of  \eqref{p} with $u_0\in \mscX(h)$ and $v$  be a solution of  \eqref{stationary}
or \eqref{finite}. Then $\mathcal{Z} (u(\cdot,t)-v(\cdot))$ is a finite and nonincreasing function of $t>0$.
\end{lem}

This lemma can be proved in a similar way as that of \cite[Lemma 2.9]{DM}
(see also \cite{An, Ma}), using a combination of Lemma~\ref{lem:zero} and Hopf lemma.

\begin{lem}\label{lem:u decreases far}
Let $u$ be a solution of  \eqref{p} with $u_0\in \mscX(h)$. Then
\begin{enumerate}
\item[\rm (1)] $u_x(x,t)<0$ for all $x>h$ and all small $t>0$;

\item[\rm (2)] $\mathcal{Z} (u_x(\cdot, t))<\infty$ for all $t>0$;

\item[\rm (3)] there exist an integer $N>0$ and a time $T>0$ such that for all $t> T$,
$u_x(\cdot,t)$ has exactly $2N-1$ zeros: $0<\xi_1 (t)<\tilde{\xi}_1(t)<\xi_2 (t)<\tilde{\xi}_2(t)<\cdots
< \tilde{\xi}_{N-1}(t)<\xi_N(t)<\infty$; moreover, $u_{xx}(\xi_i(t),t)<0$, $u_{xx}(\tilde{\xi}_i(t),t)>0$,
and $\xi_i(t),\ \tilde{\xi}_i (t)\in C^1$ for each $i$.
\end{enumerate}
\end{lem}

\begin{proof}
(1) Since $u_0 \in \mscX(h)$, there exists $X \in (0,h)$ such that $u(X,0)>0$. Then, for any $a >h$, there exists some small time $\tau >0$ such that $u (2a-X, t) < u (X,t)$ for any $t \in [0,\tau]$. Therefore, the function $w(x,t):=u(x+X,t)-u(2a-X-x,t)$ satisfies
$$
w(0,t)>0, \ w(a-X,t)\equiv 0 \ \  \forall t\in [0,\tau],\qquad w(x,0)\geqslant 0 \ \ \forall x\in [0,a-X].
$$
Using Hopf lemma on $w$, we get $w_x(a-X,t) =2u_x(a,t) <0$ for all $t\in (0,\tau]$.

One can also easily check (using again Hopf lemma), that $u_x (0,t) >0$ for all $t >0$. We can thus apply Lemma~\ref{lem:zero} and immediately obtain that (2) and the first part of (3) hold.
Finally, applying  Lemma  \ref{lem:zero} (2) to function $u_x$  we derive that $u_{xx}(\xi_i(t),t)<0<u_{xx}(\tilde{\xi}_i(t),t)$.
The $C^1$ regularity of $\xi_i (\cdot)$ and $\tilde{\xi_i} (\cdot)$ then follows from implicit function theorem for the algebraic equations $u_x(\xi_i(t),t)=0$ and $u_x (\tilde{\xi}_i (t),t)=0$.
\end{proof}


\section{The Spreading and Vanishing Cases}

\subsection{A Convergence Result in $C^2_{\rm loc}([0,\infty))$ Topology}

In this subsection we give a local uniform convergence result, which
is an analogue of \cite[Theorem 1.1]{DM} and \cite[Theorem 1.1]{DuLou}.

\begin{lem}\label{convergence} Assume {\bf(F)} and let $u$ be the solution of \eqref{p} with $u_0 \in \mscX(h)$. Then there exists a solution $v$ of \eqref{stationary} such that
\bes\label{omegalimit}
  \lim_{t\to\infty} \| u(\cdot,t)-v \|_{C^2([0,M])} =0\qquad\forall\, M>0.
\ees
In addition, $v$ cannot be a positive periodic solution.
\end{lem}

\begin{proof}
Denote by $\omega(u)$ the $\omega$-limit set of $u(\cdot,t)$ in the topology of $L^\infty_{\rm loc}
([0,\infty))$. By local parabolic estimates, the definition of $\omega(u)$ remains unchanged if
the topology of $L^\infty_{\rm loc}([0,\infty))$ is replaced by that of
$C^2_{\rm loc}([0,\infty))$.
Thanks to our Lemmas~\ref{lem:intersection} and~\ref{lem:u decreases far}, and following the ideas of Du and Matano~\cite{DM} and Du and Lou~\cite{DuLou}, one can prove that
$\omega(u)$ consists exactly of one of the solutions of \eqref{stationary}.
Moreover, according to Lemma~\ref{lem:u decreases far}, we also know that $\mathcal{Z}(u_x(\cdot,t))$ is a finite and nonincreasing function of $t$, so that $v$ cannot be a positive periodic function.
\end{proof}

\subsection{Sufficient Conditions for Spreading}

We say that, given an initial datum $u_0$, {\bf spreading} happens if $v$ in \eqref{omegalimit}
is an active state $v_* (\cdot -z)$ with $z \in \R$.
Here we give some sufficient conditions for spreading.

\begin{lem}\label{spreading} Assume {\bf(F)} and let $u$ be the solution of \eqref{p} with $u_0 \in \mscX(h)$.
Then spreading happens 
if $u_0$ satisfies  one of the following conditions:
\begin{enumerate}
\item[{\rm(1)}] for some $m\in(\theta,1)$ and $r \geqslant 0$, $u_0(\cdot)\geqslant v_m(\cdot-r)$ on $[0,\infty)$ where $v_m$ is given by~\eqref{finite} and~\eqref{extended_vm};
\item[{\rm(2)}] for some $m \in (\alpha , 1]$ and  $r\geqslant 0$,   $u_0 (\cdot) \geqslant m $ on
$[r, r +  2L(m) ]$, where $L(m)$ is a certain positive  function of $m\in(\alpha,1]$.
\end{enumerate}
\end{lem}

\begin{proof}
(1).
When $u_0(\cdot)\geqslant v_m(\cdot-r)$,  the  comparison principle implies that
$
u(\cdot,t)\geqslant v_m (\cdot-r)$ on  $[0,\infty)$ for each $t>0$.
Since active states are the only solutions of \eqref{stationary} that lie above $v_m(\cdot-r)$, the conclusion follows from Lemma~\ref{convergence}
immediately.

(2a).  First we consider the case $m \in (\theta,1]$. Choose $L(m)=L_m$ defined  in \eqnref{Lm}. If $u_0\geqslant m$ on $[r,r+2L(m)]$, then
  $u_0(\cdot) \geqslant   v_m(\cdot-r)$  so by (1),  spreading happens.

(2b). Next consider $m \in (\alpha , \theta]$. Let  $\eta(t)$  be the solution of
$$
\eta_t = f(\eta)\quad\hbox{on \ }[0,\infty),\qquad \eta(0)=m.
$$
Since $f(\cdot)>0$ in $(\alpha,1)$, with $\vep=\frac{1-\theta}{3}$ and
 $T=\int_m^{\theta+2\vep}\frac{ds}{f(s)}$ we have $\eta(T) =\theta+2\vep$.

We fix $R=L_{\theta+\vep}$.
Let $L\gg R$ be a constant to be determined and $w_0$ be a  function satisfying
\bess w_0(x)=m\hbox{\ when \ }|x|<L-1, \quad w_0(x)=0\hbox{\ when \ }|x|>L,
 \quad x w_0'(x)\leqslant 0 \ \forall\,x. \eess
Let $w(x,t)$ be the solution of  the problem
$$
 \left\{
  \begin{array}{ll}
  w_t = w_{xx} +f(w) &\forall\, x\in [-L,L], t>0,\\
  w(\pm L,t)=0 &\forall\, t>0,\\
  w(x,0)=w_0(x) &\forall\, x\in [-L,L].
  \end{array}
  \right.
$$
Let $\rho(x) =( 1+ x^2)^{-1}$. Then  $\zeta(x,t):=\rho(x)[w(x,t)-\eta(t)]$
satisfies $\zeta_t=\zeta_{xx}+4x\rho \zeta_x+[2\rho+f']\zeta$. Hence, with $Q :=2+\max_{0\leqslant s\leqslant 1} f'(s)$  we can derive that
$$
\max\limits_{|x|\leqslant L} \{\rho(x) |w(x, t)-\eta(t)|\} \leqslant e^{Qt} \cdot
\max\limits_{|x|\leqslant L} \{\rho(x) |w_0(x) -m|\}\leqslant \frac{e^{Qt} }{1+(L-1)^2}.
$$
 Taking $L=L(m)=1+\sqrt{ (1+R^2)e^{QT}/\vep-1}$ we have,  when $|x|\leqslant R$,
$$
 |w(x, T)-\eta(T)| \leqslant \frac{1}{\rho(x)} \frac{ e^{QT}}{1+(L-1)^2}
\leqslant\frac{(1+R^2)e^{QT}}{1+(L-1)^2} \leqslant \vep.
$$
Thus, $w(\cdot,T)\geqslant \eta(T)-\vep=\theta+\vep$ on $[-R,R]$.

Now if $u_0\geqslant m$ on some interval $[r, r+2L]$ for
$r\geqslant 0$, we have $u_0(x+r+L)\geqslant w_0(x)$ for all $x\in [-L,L]$,  so by comparison,  for $x\in[-R,R]$,
$
u(x+r+L,T)\geqslant w(x,T)>\theta+\vep.
$
 The assertion of the Lemma then   follows from the earlier case (2a) with $m=\theta+\vep$.
\end{proof}

We remark that the sufficient  condition (2)  originates  from Fife and McLeod \cite{FM}. The introduction of the function $\rho$ is due to Feireisl and Pol\`{a}\v{c}ik \cite{FPolacik}.

\subsection{The Threshold Phenomenon}

\begin{thm}\label{trichotomy}
Suppose that $f$ satisfies {\bf (F)} and $\phi\in \mathscr{X}(h)$ with $h>0$. Let  $u$
be the solution of~\eqref{p} with $u_0=\sigma \phi$.
Then there exist $0< \sigma_* \leqslant \sigma^* < \infty$ such that:
\begin{itemize}
\item[\rm (i)] if $\sigma >\sigma^*$, {\bf spreading} happens,  i.e., $\lim_{t\to\infty} \|u(\cdot,t)-v_*\|_{C^2([0,M])}=0\ \forall\,M>0$;

\item[\rm (ii)] if $\sigma\in[0,\sigma_*)$, {\bf vanishing} happens, i.e.,
$\lim_{t\to\infty} \|u(\cdot,t)\|_{C^2([0,\infty))}=0$

\item[\rm (iii)] if $\sigma \in [\sigma_*,\sigma^*]$, neither spreading nor vanishing happen.

\end{itemize}
\end{thm}

\begin{proof} We denote the solution of \eqref{p} by $u(x,t; u_0)$ and  define
\bes\nonumber \cA  &:=& \Big\{ \sigma\geqslant 0 \;|\;  \lim_{t\to\infty} \| u(\cdot,t;\sigma\phi)\|_{L^\infty([0,\infty))}=0\Big\},
\\ \nonumber \cB &:=& \Big\{   \sigma\geqslant 0 \;|\;  \lim_{t\to\infty} \| u(\cdot,t;\sigma\phi)-v_*\|_{L^\infty([0,M])}=0\ \forall\, M>0\Big\},
\\ \label{sigma} && \sigma_* :=\sup\{\sigma\;|\; \sigma\in\cA\},\quad\sigma^*:=\inf\{\sigma\;|\; \sigma\in\cB\}.
\ees
We remark that,
by the local parabolic estimates,
the convergence in  the definitions of $\cA$  and $\cB$ can be replaced by the $C^2([0,\infty))$ (maximum of $|u|+|u_x|+|u_{xx}|$ over $[0,\infty)$)  and $C^2_{\mathrm{loc}}([0,\infty))$ topology  respectively.

By comparison principle, if $\sigma\in \cA$, then $[0,\sigma]\in \cA$.
Also, if $\sigma\in \cA$, then there exists $T>0$  such that $\|u(\cdot,T;\sigma\phi)\|_{ L^\infty }<\alpha$. Hence, by continuity, there exists $\vep>0$ such that $\|u(\cdot,T;\hat\sigma\phi)\|_{ L^\infty }<\alpha$  for every $\hat\sigma\in[\sigma,\sigma+\vep]$. As $f<0$ in $(0,\alpha)$ we derive by comparison that  $\|u(\cdot,t;\hat\sigma\phi)\|_{ L^\infty } \to 0$ as $t\to\infty$,  so $\hat\sigma\in\cA$.    Hence, $\cA$ is open. As $0\in\cA$, we see that
$\cA=[0,\sigma_*). $

\medskip

 Similarly, if $\sigma\in \cB$, then by comparison and Lemma \ref{convergence}, $[\sigma,\infty)\in \cB$.
  In addition, since $v_*(\infty)=1$, for $L=L(\theta)$ given in Lemma \ref{spreading} (2),
  there exists $r>0$ such that $v_*>\theta$ in $[r,r+2L]$.
  As $u(\cdot,t;\sigma\phi)\to v_*$ locally uniformly, there exists $T>0$ such that
  $u(\cdot,T;\sigma\phi)> \theta$ in $[r,r+2L]$. Then by continuous dependence,
   there exists $\vep>0$ such that
  $u(\cdot,T;\hat\sigma\phi)>\theta$ in $[r,r+2L]$
  for every $\hat\sigma\in[\sigma-\vep,\sigma]$. Then by Lemma
  \ref{spreading} (2), $\hat\sigma\in \cB$. Thus, $\cB$ is open.

  Next, we show that $\cB$ is non-empty. For this, we need the technical assumption  $K:=-\inf_{s>1} f'(s)<\infty$ (without this condition, $\cB$ may be empty, c.f. \cite{DuLou, FF, Zla}).
  Comparing $u$  with  the solution of $w_t=w_{xx}-K w$ subject to the boundary condition $w(0,t)=0$ and initial datum $w(x,0)=\sigma\phi(x)$,  we find that
  \bess u(x,t; \sigma\phi) & \geqslant & w(x,t)=\sigma\int_0^h
   \frac{e^{-\frac{(x-y)^2}{4t}-K t}}{\sqrt{4\pi t}}
  \Big(1-e^{-\frac{xy}{t}}\Big)  \phi(y)\,dy. \eess
  When $\sigma$ is large enough and since $\phi$ is positive on a nontrivial interval, $u(\cdot,1;\sigma\phi)>\theta$ in $[1,1+2L(\theta)]$ so   by Lemma \ref{spreading} (2),  $\sigma\in \cB$. Hence, $\cB=(\sigma^*,\infty)$ where $\sigma^*\in[\sigma_*,\infty)$.\end{proof}

In order to prove Theorem~\ref{thm:main}, it remains to show that  $\sigma_*=\sigma^*$ and to investigate what happens in the transition case. This will
 be done  in the next section.
\bigskip


\section{The Transition Case (1)}\label{sec:conv H2}

In this section we first establish the convergence to a shifted ground state in $H^2([0,\infty))$ topology for the solution
$u(\cdot,t;\sigma \phi)$ of \eqref{p} with  $\sigma
\in [\sigma_*, \sigma^*]$. Then we complete the proof of Theorem~\ref{thm:main}, and in particular show that the threshold is sharp, that is $\sigma_* = \sigma^*$. Throughout this section, $u(x,t)=u(x,t;\sigma\phi)$ where $\sigma\in[\sigma_*,\sigma^*]$
and $\phi\in \mscX(h)$ for some $h>0$.
By Lemma~\ref{convergence}, the function $v(x)=\lim_{t\to\infty} u(x,t;\sigma\phi)$ exists and is a  solution of \eqref{stationary} (not positive periodic).  Since $\sigma\not\in \cA\cup\cB$, by Lemma \ref{limitsol},  there are only the following alternatives:
\begin{enumerate}\item[(i)]
 $v\equiv 0$; in this case
\bes\label{trans-1} \lim_{t\to\infty} \| u(\cdot,t;\sigma\phi)\|_{C^2([0,M])}=0\ \forall\, M>0,
\qquad \| u(\cdot,t;\sigma\phi)\|_{L^\infty ([0,\infty))} > \alpha\,\quad \forall\,t>0. \ees
\item[(ii)] $v=V(\cdot-z)$  with $z\in \Shift_{ground} (b)$ is a  ground  state (c.f. Lemma \ref{limitsol} (3)); in this case
\bes \label{trans-2}
 \lim_{t\to\infty} \| u(\cdot,t;\sigma\phi)-V(\cdot-z)\|_{C^2([0,M])}=0\quad\forall\, M>0.\ees
\end{enumerate}

\medskip

\subsection{Energy Estimates} In the transition case, convergence to a ground state may occur in a moving frame whose speed is not a priori known, so that we can no longer use standard parabolic estimates. To overcome this difficulty, we will use here an energy method.

For any $\psi\in H^1([0,\infty))$  we define
its energy by
\bes\label{E}
\bE[\psi]= \left\{\begin{array}{ll}
  \int_0^\infty [\psi_x^2(x)+ F(\psi(x))]dx+ \frac{1}{b}\psi^2(0)\quad&\hbox{if \ }b>0,\medskip
\\ \int_0^\infty [\psi_x^2(x)+ F(\psi(x))]dx\quad&\hbox{if \ }b=0.
 \end{array} \right.
\ees
Let $u$ be a solution of  \eqref{p}. Then $u(\cdot,t)\in H^2([0,\infty)),\ u_t(\cdot,t)\in
L^2([0,\infty))$ by Lemma \ref{lm:u to 0}. Using integration by parts and the definition
$F(u)=-2\int_0^u f(s)ds$ we have
\bess &&
\frac{d}{dt} \bE[ u(\cdot,t)] =-2\int_0^\infty u_t^2(x,t)dx\leqslant 0\quad\forall\, t>0
.\eess
In other words, \eqref{p} is a {\bf gradient flow}.

\begin{lem}\label{lem:E} There exists a constant $C$ such that
\bess
&& \int_1^\infty\int_0^\infty u_t^2(x,\tau)dxd\tau+\sup_{t\geqslant 1}\Big\{\Big| \bE[u(\cdot,t)]\Big|+ \|u(\cdot, t)\|_{H^2([0,\infty))}+\|f(u(\cdot,t))\|_{L^2([0,\infty))}\Big\}\leqslant C,
\\ && \\  &&\lim_{t\to\infty} \| u_t(\cdot,t)\|_{L^2([0,\infty))} = \lim_{t\to\infty} \| u_{xx}(\cdot,t)+f(u(\cdot,t))\|_{L^2([0,\infty))}=0.\eess
\end{lem}

\begin{proof} Note that $F>0$ in $(0,\theta)$, and that $F'(0)=0< F''(0)=-2 f'(0)$. We fix a $\gamma\in(\alpha,\theta)$. Then there exists $\vep\in(0,1)$ such that $F(s)\geqslant \vep s^2$ for $s\in[0,\gamma]$.
Consider the set
$$
J_\gamma (t):= \{x>0 : u(x,t)> \gamma\}.
$$
By Lemmas \ref{lem:zero} and \ref{lem:u decreases far}, there exists an integer $N_0$ such that
for each $t>1$,  $u(\cdot,t)$ has at most $N_0$ local maximum points.
It follows that $J_\gamma$ is the union of at most $N_0$ open intervals.
Denote by $|J_\gamma(t)|$ the length of $J_\gamma(t)$. Then since $\sigma\not\in\cB$, by Lemma \ref{spreading} (2),
we have $|J_\gamma(t)|\leqslant 2N_0 L(\gamma)$. Consequently,
$$
\int_0^{\infty} \Big(F(u)-\vep u^2\Big)dx\geqslant
 \int_{J_\gamma(t)} \Big(F(u)-\vep u^2\Big)dx   \geqslant 2\Big\{ F(1)-\vep \|u\|_{L^\infty}^2\Big\} N_0 L(\gamma)=:c_1.
$$
Since $\|u\|_{L^\infty} \leqslant\max\{1, \sigma^*\|\phi\|_{L^\infty}\}$, we see that   $c_1>-\infty$.
Hence, for $t>1$,
\begin{eqnarray*}
\bE[u(\cdot,1)] & = & 2\int_1^t\int_0^\infty u_t^2(x,\tau)dxd\tau +  \bE[u(\cdot,t)]
 \\ &\geqslant &2 \int_1^t\int_0^\infty u_t^2(x,\tau)dxd\tau + \vep \|u(\cdot,t)\|^2_{H^1([0,\infty))} +c_1 .
\end{eqnarray*}
This gives an upper bound $C$ of $\big| \bE[u(\cdot,t)] \big|$, $\|u(\cdot,t)\|_{H^1([0,\infty))}$ and $\int_1^\infty\int_0^\infty u_t^2 dxdt$.  Setting $K_0:=\max_{0\leqslant s\leqslant \|u\|_{L^\infty}} |f'(s)|$ and up to increasing $C$ we get $\|f(u(\cdot,t))\|_{L^2([0,\infty))} \leqslant K_0\|u(\cdot,t)\|_{L^2([0,\infty))} \leqslant~C$.

Next,
by \eqref{uest} and by $|f'(u)|\leqslant K_0$ we have
\begin{eqnarray*}
\frac{d}{dt} \|u_t(\cdot, t)\|^2_{L^2} & = & 2\int_0^\infty u_t[u_{xxt} +f'(u)u_t]\\
& \leqslant & -2b u_{xt}^2(0,t) - 2 \int_0^\infty u^2_{xt} dx + 2K_0 \int_0^\infty u^2_t dx\\
& \leqslant & 2K_0\|u_t(\cdot, t)\|^2_{L^2} .
\end{eqnarray*}
Hence
$\|u_t(\cdot, t)\|_{L^2}  \leqslant e^{K_0(t-s)} \|u_t(\cdot, s)\|_{L^2}$
for any  $t>s >1.$ As $\int_1^\infty \|u_t\|^2_{L^2}<\infty$, we obtain
$$\|u_t(\cdot, t)\|_{L^2} \to 0 \hbox{\ as \ } t\to \infty .$$
Finally,
from the equation $u_t =u_{xx} +f(u)$ we see that
$\|u_{xx}\|^2_{L^2} \leqslant 2 (\|u_t\|^2_{L^2} + \|f(u)\|^2_{L^2})<C$ (up to increasing $C$ again) and
$\|u_{xx}+f(u)\|_{L^2}\to 0$ as $t\to\infty$. This completes the proof of the lemma.
\end{proof}

\subsection{Convergence in Moving Coordinates}\label{subsec:moving coordinate}
By Lemma \ref{lem:u decreases far}, for any $t>T$, $u(\cdot,t)$ has exactly $N$ local maximum points
$\{\xi_i(t)\}_{i=1}^N$. In the transition case there are some $i$ such that
$u(\xi_i(t),t)>\alpha$ for all $t>T$. In what follows, denote by $\xi(t)$ the leftmost one of them:
\begin{equation}\label{def xi}
\xi(t):= \min \{ \xi_i(t) \mid u(\xi_i(t),t) >\alpha\}
\end{equation}
for all $t>T$.

Now suppose \eqref{trans-1} holds. We must have $\lim_{t\to\infty}\xi (t)=\infty$.
Indeed, if $\liminf_{t\to\infty} \xi(t)\leqslant M<\infty$, then $\limsup_{t\to\infty}
\|u(\cdot,t)\|_{L^\infty([0,2M])} \geqslant \alpha$, which contradicts the equality in \eqref{trans-1}.
 Now set
$$
y:= x-\xi(t) \quad \mbox{and} \quad w(y,t):= u(y+\xi(t),t).
$$
Lemma \ref{lem:E} implies that $w$ is bounded in $H^2([-\xi(t),\infty))$.  In addition,
 $\|w_{yy}+f(w)\|_{L^2([-\xi(t),\infty))} =\| u_t(\cdot,t)\|_{L^2([0,\infty))} \to 0$ as $t\to\infty$. Hence, there exists a sequence
$\{t_n\}_{n=1}^\infty$ in $[0,\infty)$ and a function $w_\infty\in H^2(\R)$ such that
\bess \lim_{n\to\infty} t_n=\infty,\quad\lim_{n\to\infty}   w(\cdot,t_n)= w_\infty(\cdot) \quad\hbox{\ in \ } H^2([-R,R])\quad\forall\, R>0. \eess
Furthermore,
$\|w_\infty\|_{H^2(\R)}\leqslant \sup_{t\geqslant 1}\|u(\cdot,t)\|_{H^2([0,\infty))}<\infty$ and $w_\infty$ satisfies
$$
   w''_\infty  + f( w_\infty) =0\hbox{ \ \ in \ }\R,\qquad w_\infty\in H^2(\R),\quad
     w_\infty(0)\geqslant \alpha.
$$
 Therefore $w_\infty\equiv V$.  Finally, by the uniqueness of $w_\infty(y)$ we have
\bes \label{trans1a}
\lim_{t\to\infty}\|w(\cdot,t)-V(\cdot)\|_{H^2([-M,M])}=0\quad\forall\, M>0.
\ees

Suppose \eqref{trans-2} holds. Then we have
\bes\label{trans2a} \lim_{t\to\infty} \Big(
\| u(\cdot,t)-V(\cdot-\xi(t))\|_{H^2([0,M])}+|\xi(t)-z| \Big)=0\quad \forall \, M>0.
\ees


\subsection{Uniform Convergence}
Note first that both \eqref{trans1a} and \eqref{trans2a} imply by classical embeddings that the convergence to the shifted ground state is also locally uniform in the same moving coordinates. The uniform convergence on the whole half-line relies on the following lemma on the number of maximum points of $u$:

\begin{lem}\label{lem:xiN-xi} Let $T$ be the time in Lemma \ref{lem:u decreases far}. Then
$u(\cdot,t)$ has exactly one maximum point $\xi (t)$ for any $t>T$.
\end{lem}
\begin{proof} We first claim that
\begin{equation}\label{claim_xi}
\xi_N (t) - \xi_1 (t) \to 0 \quad \mbox{ as } t \to \infty.
\end{equation}
If $\xi_N(t_n) -\xi_1(t_n)\to \xi_0 >0$ for some sequence $t_n\to \infty$, then this clearly contradicts \eqref{trans1a} or~\eqref{trans2a}.
So if our claim is not true, then
$$
\xi_N(t)-\xi_1(t) \to \infty \quad \mbox{as} \ t\to \infty.
$$
Since $u(\cdot,t)$ is strictly decreasing in $(\xi_N(t),\infty)$, there exists a large $L> \xi_N(T)$ such that
$$
u(L, T ) < u(x,T) \quad \forall \, x\in [\xi_1 (T),  L),
$$
where $T$ is the time in Lemma \ref{lem:u decreases far}. Define
$$
T_1 := \inf\{ t > T \mid \xi_N(t) = L \} \; \in (T, \infty).
$$
Then for any small $\ve>0$ we have $\xi_1(t)<L$ and $2L-\xi_1 (t) >\xi_1 (t)$
when $t\in [T,T_1+\ve]$. Set $I(t):= [\xi_1 (t), 2L-\xi_1(t)]$ and define 
$$
\zeta(x,t):= u(x,t)- u(2L -x, t)\ \mbox{ on }\  I(t)\times [T, T_1 + \ve].
$$
We will derive a contradiction below. When $t\in [T, T_1]$, by choosing $\delta>0$ sufficiently small 
we have $u(\xi_1 (t), t) > u(\xi_1 (t)+\delta, t)$, since $\xi_1(t)$ is a local 
maximum point of $u(\cdot ,t)$. On the other hand,  
$ u (2L -\xi_1(t)-\delta, t) > u(2L -\xi_1 (t), t)$ since $u(\cdot,t)$ is strictly decreasing in 
$[L, \infty)$. Therefore 
\begin{equation}\label{zeta decrease}
\zeta(\xi_1(t),t)>\zeta (\xi_1(t) +\delta,t) \quad \forall \, t\in [T,T_1].
\end{equation}
Since $\zeta(\cdot,t)$ is antisymmetric around $x= L$ on $I(t)$ and
$$
\zeta(x,T) >0 =\zeta(L,T)\quad \forall \, x\in [\xi_1 (T), L),
$$
we have $\zeta(x,t)>0$ in $x\in (\xi_1(t), L)$ as long as $\zeta(\xi_1(t),t)>0$.
Combining with \eqref{zeta decrease} we have $\zeta(\xi_1(t),t) >0$ for $t\in [T,T_1]$.
By continuity, this is true even for $t\in [T,T_1 +\ve]$ provided $\ve>0$ is small.
Consequently, $\mathcal{Z}_{I(t)} (\zeta(\cdot, t)) =1$ for all $t\in [T, T_1 +\ve]$.
By Lemma \ref{lem:zero}, such a result contradicts the fact $x=L$ is a degenerate zero of $\zeta(\cdot, T_1)$: 
$$
\zeta( L, T_1) = \zeta_x (L , T_1) = 2 u_x(L ,T_1) = 2u_x(\xi_N(T_1),T_1) =0.
$$
We now can conclude that~\eqref{claim_xi} holds.

From \eqref{trans1a} or \eqref{trans2a} (depending whether $\xi (t)$ is bounded or not), we know that $u(t,y+\xi (t)) \to V(y)$ locally uniformly as $t \to +\infty$. By standard parabolic estimates, we infer that $$u(t+1,y+\xi (t)) \to V (y)$$ in $C^2_{\rm loc}$ topology with respect to $y$. As $V'' (0)<0$, it follows that, for large $t$, $u(t+1,\cdot)$ reaches a unique local maximum in the interval $ [\xi (t) - \delta, \xi (t)+\delta]$ where $\delta >0$ only depends on $V$. Combined with~\eqref{claim_xi} and the fact that~$T$ was chosen so that the number of maximum points is constant in time, this ends the proof of the lemma.
\end{proof}

Choose now any $\varepsilon \in (0,\alpha)$ and $z_\vep >0$ such that $V( \pm z_\vep) = \vep$. From Lemma~\ref{lem:xiN-xi}, we now know that $u(y+\xi (t) + z_\vep, t) \leqslant u(\xi (t) + z_\vep,t) \to V(z_\vep)$ as $t \to +\infty$ for any $y \geqslant 0$, hence
$$\Big| u(y+\xi (t) + z_\vep,t) - V(y+z_\vep ) \Big| \leqslant 2 \vep$$
for $y\geqslant 0$ and large $t$. When $\xi$ is bounded and \eqref{trans2a} holds, it easily follows that
\bes\label{trans2b} \lim_{t\to\infty} \Big(
\| u(\cdot,t)-V(\cdot-\xi(t))\|_{L^\infty ([0,\infty))}+|\xi(t)-z| \Big)=0.
\ees
In the case when $\xi(t) \to +\infty$, that is \eqref{trans-1} and \eqref{trans1a} hold, we know thanks to the definition of $\xi$ that $u (x,t) < u(\xi (t) - z_\varepsilon,t)$ on the interval $[0,\xi (t)-z_\vep]$. Since $u(0,t) \to 0$ and $u(\xi(t) - z_\vep , t) \to V(-z_\vep) = \vep$ as $t \to +\infty$, and $f (\cdot) < 0$ in $[0,\alpha)$, it easily follows that $\limsup_{t \to \infty} \sup_{x \in [0, \xi (t) - z_\vep]} u(x,t) \leqslant \vep$. From all the above, it is straightforward to conclude that
\bes\label{trans1b} \lim_{t\to\infty}
\| u(\cdot,t)-V(\cdot-\xi(t))\|_{L^\infty ([0,\infty))}=0.
\ees

\subsection{Concentrated Compactness and  Convergence in $H^2([0,\infty))$} {\ }

We first consider  the case where \eqref{trans-1}, hence \eqref{trans1a} and \eqref{trans1b}, hold.

Let $\vep_0>0$ be a number such that $f'<0$ in $[0,\vep_0]$.
Fix  an arbitrary  $\varepsilon\in(0,\vep_0)$ and, as above, let $z_\vep>0$ be the point such that
$V(z_\vep)=\vep$.  Set $J_\vep(t):= \{ x\geqslant 0\;|\; u(x,t)\geqslant \vep\}$.
By \eqref{trans1b}, we have for $t \gg 1$ that $J_\vep (t) = [a(t),b(t)]$,
and $\lim_{t\to\infty}[\xi(t)-a(t)]=\lim_{t\to\infty}[ b(t)-\xi(t)]=z_\vep$. In addition, by \eqref{trans1a},
\bess
\lim_{t\to\infty} \| u(\cdot,t)-V(\cdot-\xi(t))\|_{H^2(J_\vep (t))}=
\lim_{t\to\infty} \|w(\cdot,t)-V(\cdot)\|_{H^2([-z_\vep,z_\vep])}=0.
\eess
Set $J^c_\vep (t) :=[0,\infty)\backslash J_\vep (t) =[0,a(t))\cup(b(t),\infty)$.
We now show that $\|u(\cdot,t)\|_{H^2(J^c_\vep (t))}$ is a small quantity.
Integrating  $uu_t=u u_{xx}+u f(u)$ over $J^c_\vep (t)$ we obtain
\bess
\int_{J^c_\vep (t)}  u u_t  dx  &=& uu_x\Big|_{0}^{a(t)}+ u u_x\Big|_{b(t)}^\infty
+ \int_{J^c_\vep (t)}\Big(u f(u)-u_x^2\Big) dx
\\ & \leqslant &  uu_x\Big|_{b(t)}^{a(t)}
-  \nu \|u\|^2_{H^1(J^c_\vep(t))}
\eess
where $\nu :=\min\{1,\min_{0\leqslant s\leqslant\vep_0}\{- f'(s)\}\}$.  Sending $t\to\infty$ and using $\int_0^\infty| uu_t|\leqslant \|u\|_{L^2} \|u_t\|_{L^2}\to 0$ we derive
\bess
\limsup_{t\to\infty}  \| u\|^2_{H^1(J^c_\vep(t))}\leqslant \lim_{t\to\infty} \frac{1}{\nu} u u_x\Big|^{a(t)}_{b(t)}=
\frac{ 2}{\nu} \Big|V(z_\vep)V'(z_\vep)\Big|. \eess
Hence,
\bess  && \limsup_{t\to\infty} \| u(\cdot,t)-V(\cdot-\xi(t))\|_{H^1([0,\infty))}
\\ &\leqslant&  \limsup_{t\to\infty} \| u(\cdot,t)-V(\cdot-\xi(t))\|_{H^1(J_\vep(t))} +
 \limsup_{t\to\infty} \| u(\cdot,t)-V(\cdot-\xi(t))\|_{H^1(J^c_\vep(t))}
  \\ &\leqslant & \Big(\frac{2}{\nu} |V(z_\vep)V'(z_\vep)|\Big)^{1/2}+ 2\| V(\cdot)\|_{H^1((z_\vep,\infty))}. \eess
 Sending $\vep\searrow 0$ we derive that $u(\cdot,t)-V(\cdot-\xi(t))\to 0$ in $H^1([0,\infty))$. Finally using
  $u_{xx}+f(u) = u_t\to 0$ in $L^2([0,\infty))$ we derive
\begin{equation}\label{convergence H2}
\lim_{t\to\infty} \Big\|u(\cdot, t) - V(\cdot-\xi(t))\Big\|_{H^2 ([0,\infty) )}=0.
\end{equation}

When \eqref{trans-2} instead of \eqref{trans-1} holds,  we have (\ref{trans2a}) and \eqref{trans2b}. A similar discussion as above (with $J_\vep(t)=[0,b(t)]$) shows that  \eqref{convergence H2} holds.
We summarize our result as follows:

\begin{lem}\label{le:Trans} Assume that $\sigma\in[\sigma_*,\sigma^*]$. Then \eqref{convergence H2}
holds for the maximum point $\xi(t)$ of $u(\cdot,t)$ (which is unique for large $t$). In addition,
either {\rm(1)} $\lim_{t\to\infty}\xi(t)=\infty$ or {\rm(2)} $\lim_{t\to\infty}\xi(t)=z$ for  some $z\in \Shift_{ground} (b)$.
\end{lem}

\subsection{The Sharp Threshold}
We are now in the position to prove the sharpness of the threshold phenomenon exhibited in the previous section.

\begin{lem}\label{eql} For each $\phi\in \mscX(h)$ with $h>0$ and $\sigma_*$, $\sigma^*$ defined in Theorem~\ref{trichotomy}, we have that~$\sigma_* =\sigma^* $.\end{lem}

\begin{proof}
By Lemma \ref{lem:xiN-xi}, we can define $\xi^*(t)$ as the unique maximum point of $u(\cdot,t;\sigma^* \phi)$ for large $t$. The function $\xi_*(t)$ is defined in the same way. By comparison, $u(\cdot,t;\sigma_* \phi)$ lies below $u(\cdot,t; \sigma^* \phi)$ and, using Lemma \ref{le:Trans}, we can infer that $\lim_{t\to\infty} |\xi^*(t)-\xi_*(t)|=0$.

We proceed by contradiction and assume that $\sigma_* < \sigma^* $. Then, following the same argument as for Lemma 4.5 in \cite{DM}, there exist positive constants $t_0,\delta$ and~$\epsilon $ such that
\begin{equation}\label{initial}
u(x,t_0+\delta; \sigma^* \phi)>u(x-a,t_0; \sigma_* \phi)\mbox{   for }x>a, \ 0\leqslant a\leqslant \epsilon.
\end{equation}
In particular, if $a=0$, we get $u(0,t+\delta;\sigma^* \phi)\geqslant u(0,t;\sigma_*\phi)$ for $t\geqslant t_0$.

Moreover, by Lemma~\ref{le:Trans} and thanks to the Robin boundary condition, there exists a small positive constant $\gamma$ such that $\gamma < \epsilon$ and $\xi^* (t) > \gamma$ for all $t > t_0 +\delta$. Consequently,
\begin{equation}\label{boundary}
u(\gamma ,t+\delta; \sigma^* \phi)> u (0,t +\delta; \sigma^* \phi) \geqslant u(0,t;\sigma_* \phi)\mbox{   for all }t\geqslant t_0.
\end{equation}
Combining \eqref{initial} and \eqref{boundary}, we have $u(x,t+\delta;\sigma^* \phi)>u(x-\gamma,t;\sigma_* \phi)$ for all $t\geqslant t_0$ and $x\geqslant \gamma$. Using again Lemma~\ref{le:Trans}, it follows that $\lim_{t\to\infty} |\xi^*(t)-\xi_*(t)-\gamma|=0$. Having reached a contradiction, we have proved the lemma.
\end{proof}

Note that the sharp threshold value $\sigma^*$, which is now well defined, depends on both the initial datum $\phi \in \cup_{h>0} \mscX (h)$ and the boundary condition parameter $b \geqslant 0$. We denote it here by $\sigma^* (\phi,b)$ and in the following theorem, which will end the proof of Theorem~\ref{thm:main}, describe its properties as a function of $\phi$ and $b$.
\begin{thm}\label{contsigma} The function $(\phi,b) \mapsto \sigma^* (\phi,b)$ is continuous in $\cup_{h>0}\mscX(h) \times [0,\infty)$ under the $L^{\infty} \times \R$-topology, and is nonincreasing with respect to both $b$ and $\phi$.
  \end{thm}

  \begin{proof} In this proof, we will denote by $u(\cdot,\cdot;u_0;b)$ the solution of \eqref{p}, so that the dependance on $b$ of the solution also appears explicitly.

First, note that the monotonicity is an immediate consequence of the comparison principle. Next, fix $(\phi_0,b_0) \in \cup_{h>0}\mscX(h) \times [0,\infty)$, and consider a sequence $(\phi_n ,b_n)_{n=1}^\infty $ such that, for any $n \in \mathbb{N}$, $(\phi_n, b_n) \in \cup_{h>0}\mscX(h) \times [0,\infty)$ and, as $n \to \infty$, $b_n \to b_0$ and $\phi_n \to \phi_0$ uniformly.

For any fixed $\vep>0$,
   $\lim_{t\to\infty} u(\cdot,t;[\sigma^*(\phi_0 ,b_0 )+\vep]\phi_0;b_0)=v_*(\cdot)$ in $C^2_{\mathrm{loc}}([0,\infty))$. Let  $r>0$ be a big number such that $v_*>\theta$ in $[r,r+2L(\theta)]$.
  Then there exists $T>0$ such that $u(\cdot,T; [\sigma^*(\phi_0,b_0)+\vep]\phi_0;b_0)>\theta$ in $[r,r+2L(\theta)]$.
  Consequently, there  exists $N>0$ such that for each $n>N$,   $u(\cdot,T; [\sigma^*(\phi_0,b_0)+\vep]\phi_n ;b_n )>\theta$ in $[r,r+2L(\theta)]$. Then by Lemma \ref{spreading} (2), $\sigma^*(\phi_n ,b_n )\leqslant \sigma^*(\phi_0,b_0)+\vep$. Thus, $\limsup_{n\to\infty} \sigma^*(\phi_n ,b_n ) \leqslant \sigma^*(\phi_0,b_0)$.
  \smallskip

  Next, for any fixed $\vep\in(0,\sigma_*(\phi_0,b_0))$, $\lim_{t\to\infty} \|u(\cdot,t;[\sigma_*(\phi_0,b_0)-\vep]\phi_0 ;b_0 )\|_{L^\infty([0,\infty))}=0$. Hence, there exists $T>0$ such that
  $u(\cdot,T;[\sigma_*(\phi_0 , b_0)-\vep]\phi_0 ;b_0)\leqslant\alpha/3$. Consequently, there exists $N_1>0$ such that
  for every $n>N_1$,   $u(\cdot,T;[\sigma_*(\phi_0,b_0)-\vep]\phi_n ;b_n )\leqslant\alpha/2$. Since $f<0$ in $(0,\alpha)$,
  this implies that $\lim_{t\to\infty}  \|u(\cdot,t;[\sigma_*(\phi_0,b_0)-\vep]\phi_n ;b_n)\|_{L^\infty} =0$ so
  $\sigma_*(\phi_n,b_n)\geqslant \sigma_*(\phi_0 ,b_0)-\vep$. Thus, $\liminf_{n\to\infty} \sigma_*(\phi_n ,b_n )\geqslant\sigma_*(\phi_0,b_0)$.

  Finally, since for each $\phi\in \cup_{h>0}\mscX(h)$ and $b\geqslant 0$, $\sigma^*(\phi ,b)=\sigma_*(\phi ,b)$, we derive that
  $\sigma^*(\phi ,b )=\sigma_*(\phi ,b )$ is continuous in $\cup_{h>0} \mscX(h) \times [0,\infty)$ under $L^\infty \times \R$ topology.
  \end{proof}

\section{The Transition Case (2)}

This section is devoted to the proof of Theorem~\ref{thm:main2}, which is equivalent to the combination of the three lemmas below. The first one proves that the motion $\xi (t)$ is asymptotically slow. This holds without any additional regularity assumption on the nonlinearity~$f$.

\begin{lem}\label{o(t)?}
Let $\xi (t)$ be defined as in Theorem~\ref{thm:main}. Then it can be replaced by another $C^1$ function
$\tilde{\xi} (t)$ with $\tilde{\xi} '(t) \to 0$, $\tilde{\xi}  (t) = o(t)$ and $\tilde{\xi}(t)-\xi(t) = o(1)$
as $t \to +\infty$.
\end{lem}
\begin{proof}
The argument is similar to the one we used in Lemma~\ref{lem:xiN-xi}. From the previous section, we know that
$u(y+\xi (s), t+s)  \to V(y)$ as $t \to +\infty$ locally uniformly with respect to both~$s$ and~$y$. It easily follows that
$$\alpha (t) := \sup_{s \in [-1,2]} | \xi (t+s) - \xi (t) | \to 0 \quad \mbox{ as } t \to +\infty.$$
One can then construct a $C^1$ function $\tilde{\xi}$ so that, for any $n \in \mathbb{N}$, $\tilde{\xi} (n)= \xi (n)$ and $|\tilde{\xi} '(t)| \leqslant 2 \alpha (n)$ on $[n,n+1]$. Then $\tilde{\xi} ' (t) \to 0$ and $\tilde{\xi} (t) - \xi (t) = o(1)$ as $t\to +\infty$. The latter implies that the convergence in~(iii) of Theorem~\ref{thm:main} still holds with $\tilde{\xi}$ instead of $\xi$. The fact that $\tilde{\xi} (t)=o(t)$ is an immediate consequence, which ends the proof of the lemma.
\end{proof}

In the previous section, we have shown that in the transition case, the convergence to a ground state could take place with either finite or infinite shift. We now provide some simple criteria, depending on $f$ and $b$, such that any of the two occurs.

\begin{lem}\label{sufficient and necessary} For any $\phi \in \cup_{h >0}\mscX (h)$, define $\xi(t)$ as in Theorem~\ref{thm:main}. Then
\begin{enumerate}
 \item[{\rm(1)}] $\lim\limits_{t\to \infty} \xi (t)=z\in \Shift_{ground} (b)$ for some $\phi \in \cup_{h>0} \mscX (h)$ $\Longleftrightarrow$ $b=\frac{s_0}{\sqrt{F(s_0)}}$ for some $s_0\in (0,\theta )$; \medskip

 \item[{\rm(2)}]$\lim\limits_{t\to \infty} \xi (t)=\infty $ for some $\phi \in \cup_{h>0}\mscX (h)$ $\Longleftrightarrow$ there exists a small $\epsilon >0 $ such that  $b<\frac{s}{\sqrt{F(s)}}$ for all $s\in (0,\epsilon )$.
\end{enumerate}
\end{lem}

\begin{proof}
(1)The $`` \Rightarrow "$ part is obvious by part (3) of Lemma \ref{limitsol}.

For the $`` \Leftarrow "$ part, we denote $z_0:=$inf$\{z>0\mid z\in  \Shift_{ground} (b)\}$. It is obvious that $z_0>0$ and $V(\cdot-z_0)$ is a stationary solution of \eqref{p}. For any $\rho >0$, we can construct an initial datum~$\phi_0$ such that
\begin{equation}
\begin{cases}
\phi_0(x)\equiv V(x-z_0), & x\in [0,2z_0],\\
\phi_0(x)\equiv 0, & x\in [2z_0+\rho,\infty),\\
\phi_0'(x)<V'(x-z_0)<0, & x\in (2z_0,2z_0+\rho],\\
\phi_0\in C^1([0,2z_0+\rho]).\\
\end{cases}
\end{equation}

We first show that $u(\cdot,t;\phi_0)$ converges to $0$. Indeed, $u(\cdot,t;\phi_0)<V(\cdot-z_0)$ for all $t>0$ by comparison principle.
In particular, $u(0,t;\phi_0)<V(-z_0)$ for all $t>0$. Noting the decay rates of $u$ and $V$ at infinity we have $u(\cdot,1;\phi_0)<V(\cdot-z_0 +\epsilon)$ for some small $\epsilon>0$, then $u(\cdot,t;\phi_0)<V(\cdot-z_0 +\epsilon)$ for all $t>1$ since $V(x-z_0-\epsilon)$ is a
stationary solution of the equation in \eqref{p} and
$$
u(0,t;\phi_0)< V(x-z_0)|_{x=0} <V(\cdot-z_0 +\epsilon)|_{x=0} \mbox{ for all } t>1.
$$
Therefore, $u(\cdot,t;\phi_0)$ converges to $0$ rather than $V(\cdot-z)$ for any $z\in \Shift_{ground}(b)$.
This implies that $\sigma^*(\phi_0)>1$.
Because of the fact that $\phi_0'(\cdot)<V'(\cdot-z_0)$ for $x\in (2z_0,2z_0+\rho]$, we know that $\sigma^*(\phi_0)\phi_0$ has exactly
one intersection point with $V(\cdot-z_0)$ in $[0,\infty)$.

If $\lim\limits_{t\to \infty} \|u(\cdot,t;\sigma^*(\phi_0)\phi_0) - V(\cdot - \xi(t))\|_{H^2([0,\infty))}=0$ with $\lim\limits_{t\to \infty}\xi(t)=\infty $, then it follows that $\lim\limits_{t\to \infty}u(0,t;\sigma^*(\phi_0)\phi_0)=0$.
By the zero number properties: Lemmas \ref{lem:zero} and \ref{lem:intersection}, there exists a time $T_0$ such that
the unique zero point $x(t)$ of $u(\cdot,t;\sigma^*(\phi_0)\phi_0) - V(\cdot-z_0)$ becomes degenerate at $t=T_0$ and
will disappear after $T_0$. Hence $u(x,t;\sigma^*(\phi_0)\phi_0)<V(x-z_0)$ for all $t>T_0$.
This contradicts the fact that $\xi(t) \to \infty$, and we conclude that $\xi (t) \to z \in \Shift_{ground} (b)$.

\medskip

(2) We first prove the $`` \Rightarrow "$ part. There exists an initial datum $\phi_1$ such that $\lim_{t\to \infty} \|u(\cdot,t;\phi_1) - V(\cdot - \xi(t))\|_{H^2([0,\infty))}=0$ with $\xi(t) \to \infty\ (t\to \infty)$. If the conclusion does not hold, then there exists a sequence $s_n\rightarrow 0$ such that $b\sqrt{F(s_n)}\geqslant s_n $ for all $n\in \mathbb{N}$. A consequence is that $b>0$, and so
$u(0,1;\phi_1)= bu_x(0,1;\phi_1) >0$ by Hopf lemma.

Take $x_n >0$ be such that $V(-x_n)= s_n$, then for each~$n$, $V' (-x_n) = \sqrt{F(V(-x_n))}=\sqrt{F(s_n)}$,
hence $V (-x_n) \leqslant b V' (-x_n)$. This implies that $V(x-x_n)$ is a lower solution of \eqref{p} for each~$n$.
Since $x_n\rightarrow \infty $ as $s_n\to 0$ and $u_x (\cdot,1;\phi_1) <0$ for large $x$, we can take a large $N$ such that
$$
u(x,1;\phi_1) > V(x-x_N) \mbox{ in } [0, \bar{x}) \mbox{ for some } \bar{x}>0,
\quad \mbox{and} \quad  u(x,1;\phi_1) < V(x-x_N) \mbox{ in } (\bar{x}, \infty).
$$
If $V$ satisfies the Robin boundary condition $V(-x_N) = b V' (-x_N)$, then Lemma~\ref{lem:intersection} applies. If not, the zero number argument can still be used to prove that $\mathcal{Z} (u(\cdot,t;\phi_1) - V (\cdot-x_N)) \leqslant 1$ for all $t \geqslant 1$. Indeed, proceed by contradiction and assume that $V(-x_N) < b V' (-x_N)$ and that there exists
$$T_1 = \inf \{ t \geqslant 1 \; | \  \mathcal{Z} (u(\cdot,t;\phi_1) - V (\cdot-x_N)) \geqslant 2 \} \in ( 1, \infty).$$
As Lemma~\ref{lem:zero} applies on any interval $[\delta,\infty)$ with $\delta >0$, the only possibility is that intersections appear from the boundary, and  $u (0,T_1;\phi_1) = V (-x_N)$. From the boundary conditions, we also have $u_x (0,T_1;\phi_1) = \frac{1}{b} u (0,T_1;\phi_1) < V' (-x_N)$ and thus (recalling, for instance, the decay rates of $u$ and $V$ as $x \to \infty$), $u (\cdot, T_1;\phi_1) < V(\cdot - x_N)$
in $(0,\infty)$. For two intersections to appear on the boundary from time $T_1$, the unique zero point of $u (\cdot,T_1;\phi_1) - V (\cdot - x_N)$ must be degenerate, which is not the case. We conclude that $\mathcal{Z} (u(\cdot,t;\phi_1) - V (\cdot-x_N)) \leqslant 1$.

Finally, we also know that $u(0,t;\phi_1) < V( - x_N)$ for any large $t$. It follows that $u(\cdot,t;\phi_1) < V (\cdot - x_N)$ in the whole domain $[0,\infty)$ for any large $t$. This, however, contradicts our assumption $u(\cdot,t;\phi_1)\to V(\cdot- \xi(t))$ with $\xi(t)\to \infty$.

For the $''\Leftarrow''$ part, we divide it into two cases. The first case is that $b<\frac{s}{\sqrt{F(s)}}$ for all $s\in (0,\theta )$. In this case, it is obvious that $\lim\limits_{t\to \infty} \xi(t)=\infty $ for any initial datum $\phi$.

The second case is $b \sqrt{F(s_0)} =s_0$ for some $s_0\in(\epsilon,\theta)$. Without loss of generality, we assume $s_0:= \max \{s\mid b \sqrt{F(r)} < r \mbox{ for all } r\in (0,s)\}$. Choose $x_0>0$ such that $V(-x_0)=s_0$. It is easily seen from the phase plane that
for any $m$ with $m-\theta>0$ small, the trajectory passing through $(m,0)$ (lying close to that of $V$) also intersects with
the line $v=bv'$ at some point $(s_1, \frac{s_1}{b})$ with $s_1 <s_0$.
Denote as in \eqref{finite} the compactly supported stationary solution corresponding to this trajectory by $v_m$. Assume $x'_m>0$ satisfies
$v_m(x'_m) =bv'_m (x'_m)=s_1$, then $v_m (x'_m)<V(-x_0)$.
It is easy to construct an initial datum $\phi_2$ such that $\phi_2 \equiv 0$ in the support $[0, 2L_m -x'_m]$ of $v_m (\cdot + x'_m)$ and, up to multiplication by the threshold parameter $\sigma^*$, such that the solution $u(\cdot,t;\phi_2)$ converges to $V(\cdot-\xi(t))$ for some $\xi(t)$. In particular, since $v'_m (2L_m)<0$, the solution $u(\cdot,t;\phi_2)$ intersects $v_m (\cdot + x'_m)$ at exactly one point for small times. If $\xi(t)\to z\in \Shift_{ground}(b)$, then $\lim_{t\to \infty}u(0,t;\phi_2)\geqslant V(-x_0)>v_m (x'_m)$. Therefore, there exists a time $T_2$ such that $u(0,t;\phi_2)$ goes up across $v_m (x'_m)$ at time $T_2$. Applying the zero number argument on $[0,2L_m - x'_m]$ and reasoning as above, we have $u(\cdot,t;\phi_2)> v_m (\cdot+x'_m)$ for $x\in [0,2L_m -x'_m]$ and $t>T_2$. By Lemma~\ref{spreading}, spreading happens for $u$, a contradiction. Then we must have $\lim_{t\to \infty}\xi(t)=\infty$ for initial datum $\phi_2$.
\end{proof}

The last lemma completes the proof of Theorem~\ref{thm:main2}.

\begin{lem}
$\Sigma _1$ is a closed set of $\Sigma $ in $L^{\infty}$-topology.
\end{lem}

\begin{proof}
If $\Sigma_1$ is not closed, then we can find an initial datum $\phi_0\in \Sigma\backslash \Sigma_1$ and a
sequence $\{\phi_n\}_{n\in\mathbb{N}} \subset \Sigma _1$ such that $\lim_{n\to \infty}\|\phi_n-\phi_0\|_{L^\infty}=0$.

By Lemma~\ref{lem:u decreases far} there exists $L_1>0$ such that, for $x>L_1$, we have $u_x(x,1;\phi_0)<0$. Because $\phi_0 \notin \Sigma _1$, there exists $\xi (t) \to \infty$ such that $\lim\limits_{t\to \infty} \|u(\cdot,t;\phi_0) - V(\cdot - \xi(t))\|_{H^2([0,\infty))}=0$. Using Lemma~\ref{lem:xiN-xi}, we can find two constants $T_2 > 1$ and $L_2 > L_1$ such that $u(x,T_2;\phi_0)<\frac{1}{2}u(x,1;\phi_0)$ for $x \in [0,L_1]$ and $u_x(x,T_2;\phi_0)>0$ for $x\in [0,L_2]$. By Lemma~\ref{o(t)?} and up to increasing $T_2$ and $L_2$, we can also get $u(L_2,t+(T_2 -1);\phi_0)>\alpha > u(L_2,t;\phi_0)$ for all $t \in [0,1]$. By comparison principle on $[L_2, \infty)$ (where $\phi_0 \equiv 0$), we get that $u(x,T_2;\phi_0) > u (x,1;\phi_0)$ for $x \in [L_2,\infty)$. We conclude that $u(\cdot,T_2;\phi_0)$ and $u(\cdot,1;\phi_0)$ intersect only once, and the (non-degenerate) intersection is located in the interval~$[L_1,L_2]$.

Because $u(x,t;\phi_n)$ converges as $n\to \infty$ to $u(x,t;\phi_0)$ in $C^1_{\rm loc}$ topology with respect to
both $t$ and $x$, we get that for $n$ large enough, $u(\cdot,T_2;\phi_n)$ and $u(\cdot,1;\phi_n)$ also intersect only once. Indeed, $u(\cdot,T_2;\phi_n)$ lies below $u(\cdot,1;\phi_n)$ on $[0,L_1 ]$, $u_x (\cdot,T_2;\phi_n) > 0 > u_x (\cdot,1;\phi_n)$ on $[L_1, L_2]$ and $u(L_2,t+(T_2 -1);\phi_n) > u(L_2,t;\phi_n)$ for all $t \in [0,1]$. Reasoning as above, we infer as announced that the intersection is unique.

Lastly, since $\phi_n \in \Sigma_1$, there must exist some time $T_3$ such that $u(0,1+T_3;\phi_{n})=u(0,T_2+T_3;\phi_{n})$. Using the zero number argument, $u(\cdot,1+T_3;\phi_n )< u(\cdot , T_2 + T_3 ;\phi_{n})$ for $n$ large enough. By the comparison principle and the convergence to a shifted ground state $V(\cdot -z)$, we can get a contradiction as in the proof of Lemma~\ref{sufficient and necessary} (1).
%
\end{proof}

\section{Asymptotic Behavior of $\xi(t)$}\label{estimate}

Throughout this section we assume that \eqref{convergence H2} holds with $\xi(t)\to\infty$ as $t\to\infty$. In particular, according to Theorem~\ref{thm:main2} and Remark~\ref{zground}, the set $Z_{ground}(b)$ must be either empty or bounded. We would like to know the asymptotic behavior of $\xi(t)$, which was an open problem even in the case $b=0$~\cite{FF}, and more precisely to prove Theorem~\ref{thm:main3}.
Most calculations in this section
deal with the case $b\lambda < 1$, while the analogue results in case $b\lambda =1$ and $f$ satisfying {\bf (Fk)}
with $k=2$ will be presented directly, since the proof is similar.
From now on, we use the notation \bess (\phi,\psi):=\int_0^\infty \phi(x)\psi(x)dx,\quad \|\phi\|=\sqrt{(\phi,\phi)},
\quad \|\phi\|_\infty=\|\phi\|_{L^\infty([0,\infty))}. \eess

\subsection{The Center Manifold}

We begin by a lemma about eigenvalues of the linearized problem around the limiting ground state on the whole line.

\begin{lem}\label{le1} Let $V$ be the unique even positive solution of $V''+f(V)=0$ in $\R$ subject to $V(\infty)=0$. Set
$\cL_0 \varphi =\varphi'' + f'(V)\varphi$, and consider the eigenvalue problem
\bess
\cL_0 \varphi=\mu \varphi \hbox{ \ in \ }L^2(\BbbR).
\eess
\begin{enumerate}
\item[\rm (1)] The principal eigenvalue of $\cL_0$, denoted by $\mu_1$, is positive.
Its associated principal eigenfunction is even, and can be normalized by the
condition $\varphi^0_1>0$ and $\|\varphi^0_1\|_{L^2(\BbbR)}=1$.

\item[\rm (2)] $\mu_2=0$ is the second eigenvalue and $\varphi^0_2=V'/\|V'\|_{L^2(\BbbR)}$
is the associated eigenfunction;

\item[\rm (3)]
 The following number  $\mu_3$ is negative
\bess
  \mu_3 & := &- \inf_{\int_{\R} \varphi^2 =1, \int_{\R} \varphi \varphi^0_1 =0,
 \int_{\R} \varphi \varphi^0_2 =0 } \int_{\BbbR} [\varphi'{}^2+ f'(V) \varphi^2].
\eess
\end{enumerate}
\end{lem}

\begin{proof}
Differentiating $V''+f(V)=0$ we have $\cL_0 V'=0$. Hence,
$(0,V')$ is an eigenpair of the operator $\cL_0$.  As $V'$ changes sign exactly once and $f'(0)<0$,
$0$ is the second eigenvalue, so the principal eigenvalue $\mu_1$ is positive, and the
principal eigenfunction $\varphi^0_1$ can be taken as positive and even (as $V$ is even).
In addition, since $f'(0)<0$, we have the alternative that either there is a third eigenvalue
$\mu_3\in (f'(0),0)$ or the remaining spectrum lies in $(-\infty, f'(0)]$ in which case $\mu_3= f'(0)$.
\end{proof}

Let us now stress that, considering the reaction-diffusion equation on the whole real line, the set of the shifted $V$ is a center manifold of any of its element. The fact that \eqref{convergence H2} holds with $\xi(t)\to\infty$ as $t\to\infty$ roughly means that the solution is moving along this manifold.

In order to describe this motion, we first introduce an approximated center manifold such that any of its elements satisfy the Robin boundary condition. Let us first define
\bess &  B(\xi) := \frac{ V(\xi)+b V'(\xi)}{1+b\lambda },\quad
  \Phi(x,\xi) := V(\xi-x)-  B(\xi) e^{-\lambda x},
\\ & R(x,\xi) := \Phi_{xx}(x,\xi)+f(\Phi(x,\xi)),
\qquad \cL^\xi \varphi := \varphi_{xx} +f'(\Phi(\cdot,\xi)) \varphi.\eess
By Theorem~\ref{thm:main2} and Remark~\ref{zground}, $B(\xi)>0$ when $\xi>0$ is large enough.
We call
$$\cM := \{ \Phi(\cdot, \xi)\;|\; \xi\geqslant 0\}$$
the {\bf approximated center manifold} and, as announced, for each $\varphi \in \cM$, $\varphi=\Phi(\cdot,\xi)$ satisfies the boundary condition
\bess
  \varphi -b \varphi_x \Big|_{x=0} = V(\xi)-B(\xi)-b[- V'(\xi)+\lambda B(\xi) ] =0.
\eess
Moreover, when \eqref{convergence H2} holds with $\xi(t)\to\infty$ as $t\to\infty$
we have
\bess
  \lim_{t\to\infty} \|u(\cdot,t)-\Phi(\cdot,\xi(t))\|_{H^2([0,\infty))}=0.
\eess
 We shall follow the work of Carr and Pego~\cite{CP},  Fusco and Hale~\cite{FuscoHale},
 Alikakos, Bates and Fusco~\cite{ABF}, Alikakos and Fusco~\cite{AF}, and Chen et. al.~\cite{Chen1,Chen2} to study the evolution of  $\xi(t)$. The main idea will be to prove that the behavior of $\xi (t)$ is mostly dictated by the gap between the approximated center manifold and the shited ground states, and in particular by the remainder~$R$.

\subsubsection{The Remainder $R$} For notational simplicity, we write
 $\Phi=\Phi(x,\xi),\ R=R(x,\xi),$ $V=V(\xi-x)$, and $W=B(\xi) e^{- \lambda x}$.
  Then  $\Phi=V-W$, $V_{xx}+f(V)=0$, and $W_{xx}=\lambda^2 W$. Using $\lambda^2= -f'(0)$
  we obtain, when $x\geqslant0$ and $\xi\geqslant 0$,
\bess  R &=& V_{xx}-W_{xx} + f(V-W)= - f(V) + f'(0)W+f(V-W)
\\ &=& W \int_0^1 [f'(0)-f'(V-s W)]\,ds
 \\ &=&  O(1) [V  + W] W= O(1)[ e^{-\lambda |x-\xi|} + e^{-\lambda (\xi+x)}] e^{-\lambda(\xi+x)}
 = O(1) e^{-2\lambda\max\{\xi,x\}},
 \\ R_\xi &=& W_\xi \int_0^1 [f'(0)-f'(V-s W)]\,ds -W \int_0^1 f''(V-sW)[V_\xi-sW_\xi]ds
 \\ &=& O(1)[| W_\xi|(V+W)+W (|V_\xi|+|W_\xi|)]= O(1) e^{-2\lambda\max\{\xi,x\}}\;. \eess
Hence,
\bess
\|R(\cdot,\xi)\| = O(1) \sqrt{1+\xi} e^{-2\lambda \xi}, \quad
\|R_\xi (\cdot,\xi)\| =O(1)\sqrt{1+\xi} e^{-2\lambda \xi}.
\eess
Also, for the energy $\bE$ defined in (\ref{E}), we have, since $\Phi(0,\xi)=b\Phi_x(0,\xi)$ and $\Phi(\infty,\xi)=0$,
\bess \frac{d}{d\xi} \bE[\Phi(\cdot,\xi)]
&=& 2\int_0^\infty[ -\Phi_{xx} - f(\Phi)]\Phi_{\xi}\,dx = -2(R,\Phi_\xi),
 \\ ( R,\Phi_\xi) &=& ( R, \Phi_\xi+\Phi_x) -\int_0^\infty \Big[ \Phi_{xx} + f(\Phi)\Big]\Phi_x \,dx
\\ &=& [\lambda B(\xi)-B'(\xi)] \int_0^\infty R(x,\xi)e^{-\lambda x}dx
  +\frac12 \Phi_x^2(0,\xi)-\frac12 F(\Phi(0,\xi))
  \\ &=& O(e^{-3\lambda\xi}) +\frac12\Phi_x^2 (0,\xi) -\frac{\lambda^2}2 \Phi^2(0,\xi)+ O(\Phi^3(0,\xi))
   \\ &=& \frac 12\Big\{ \Big(\lambda B(\xi)-V'(\xi)\Big)^2-{\lambda^2} \Big(V(\xi)-B(\xi)\Big)^2\Big\}+O(e^{-3\lambda\xi})
     \\ &=& \lambda B(\xi)[\lambda V(\xi)-V'(\xi)]+O(e^{-3\lambda \xi})
      \\ &=& \frac{2\lambda^2(1-b\lambda ) A^2}{1+b\lambda} e^{-2\lambda\xi} + O(1) e^{-3\lambda \xi}
\eess
by \eqref{asmpV}.
 We summarize our calculation as follows:

\begin{lem} For  $\xi\geqslant 0$ and $b\geqslant 0$, $B(\xi)=\frac{1-b\lambda}{1+b\lambda} Ae^{-\lambda\xi}+O(e^{-2\lambda\xi})$ and
\bess &
\|R(\cdot,\xi)\|_\infty+\|R_\xi(\cdot,\xi)\|_\infty = O(1) e^{-2\lambda \xi},\quad \|R(\cdot,\xi)\|+\|R_\xi(\cdot,\xi)\|=O(1)
\sqrt{1+\xi}e^{-2\lambda\xi}, \\
&\displaystyle  -\frac12\frac{d}{d\xi} \bE[\Phi(\cdot,\xi)]= ( R(\cdot,\xi),\Phi_\xi(\cdot,\xi))
= \frac{2\lambda^2(1-b\lambda ) A^2}{1+b\lambda} e^{-2\lambda\xi} + O(1) e^{-3\lambda\xi}.
\eess
When $b\lambda =1$ and $f\in C^3$, the next order  expansions are
$B(\xi)= \frac{f''(0)A^2}{12 \lambda} e^{-2\lambda \xi} +O(e^{-3\lambda \xi})$ and
\bess & \|R(\cdot,\xi)\|_\infty+\|R_\xi(\cdot,\xi)\|_\infty = O(1) e^{-3\lambda\xi},\quad
 \|R(\cdot,\xi)\|+\|R_\xi(\cdot,\xi)\|=O(1)
\sqrt{1+\xi}e^{-3\lambda\xi},
\\  &  ( R(\cdot,\xi),\Phi_\xi(\cdot,\xi))
=  \frac{f''(0) A^3}{6} e^{-3\lambda\xi} + O(1) e^{-4\lambda\xi}.
\eess\end{lem}

\subsubsection{The Approximated Eigenvalue Problem}
Recall that
 \bess \cL^\xi \varphi =\varphi_{xx} + f'(\Phi(\cdot,\xi))\,\varphi\eess
and consider the eigenvalue problem
\bess
  \cL^\xi \varphi=\mu\varphi,  \quad \varphi \in
 \{\eta \in H^2([0,\infty))\;|\; \eta(0)=b \eta'(0),\ \eta(\infty)=0\}.
 \eess
When $b>0$, it is associated with the bilinear form
\bess
  \langle \varphi,\psi\rangle^\xi &:=& \int_0^\infty \Big( \varphi_x\psi_x
 - f'(\Phi(\cdot,\xi)) \varphi\psi\Big)dx +\frac{1}{b} \varphi(0)\psi(0)
\\ &=&-(\varphi,\cL^\xi\psi)+\frac{\varphi(0)}{b} [\psi(0)-b\psi_x(0)]
  =-(\cL^\xi\varphi,\psi) +\frac{\psi(0)}{b}[\varphi(0)-b\varphi_x(0)]
\eess
and, when $b=0$, it is associated with the bilinear form
\bess
 \langle \varphi,\psi\rangle^\xi :=  \int_0^\infty \Big( \varphi_x\psi_x
 - f'(\Phi(\cdot,\xi)) \varphi\psi\Big)dx = -(\varphi,\cL^\xi\psi)  =-(\cL^\xi\varphi,\psi) .
\eess

\begin{lem} For each $\xi>0$,  there exist two eigenpairs $(\nu_1(\xi),\varphi_1(\cdot,\xi))$
and $(\nu_2(\xi),\varphi_2(\cdot,\xi))$ such that $\cL^\xi\varphi_i(\cdot,\xi)=
\nu_i(\xi)\varphi_i(\cdot,\xi)$ $(i=1,2)$ and
\bess
\nu_1(\xi):=-\min_{\|\varphi\|=1} \langle \varphi,\varphi\rangle^\xi= -\langle\varphi_1,
\varphi_1\rangle^\xi=(\cL^\xi \varphi_1,\varphi_1),\quad \|\varphi_1\|=1,\ \varphi_1>0, \\
\nu_2(\xi):=-\min_{\|\varphi\|=1,\varphi\perp\varphi_1} \langle \varphi,\varphi\rangle^\xi= -
\langle\varphi_2,\varphi_2\rangle^\xi =(\cL^\xi \varphi_2,\varphi_2),\quad  \|\varphi_2\|=1,\  \varphi_2'(0)> 0.
\eess
 We define
\bess
  \nu_3(\xi):= -\inf_{\|\varphi\|=1,\varphi\perp\varphi_1,\varphi\perp\varphi_2} \langle \varphi,\varphi\rangle^\xi.
\eess
Then, with $(\mu_1,\varphi^0_1(x)), (\mu_2,\varphi^0_2(x))$ and $\mu_3$ as in Lemma \ref{le1}, we have
\bess
  & \lim\limits_{\xi\to\infty} \nu_i(\xi) =\mu_i \ \ (i=1,2,3),  \\
   & \lim\limits_{\xi\to\infty}\| \varphi_i(\cdot,\xi) -\varphi^0_i(\cdot-\xi)\| =0\ \ (i=1,2).
\eess
\end{lem}

\subsection{Flow Along the Approximated Center Manifold}
\subsubsection{The Tubular Neighborhood}
 Given a function $u\in L^2([0,\infty))$,  let
\bess
 d(u,\cM):=\min_{\varphi\in \cM}\|u-\varphi\|^2 , \quad d(\xi,u):= \|u(\cdot)-\Phi(\cdot,\xi)\|^2.
\eess
We can calculate
\bess  \frac{d}{d\xi} d(\xi,u) &=& 2(\Phi_\xi, \Phi-u),
\\
\frac{d^2}{d\xi^2}d(\xi,u) &=& 2\|\Phi_\xi (x,\xi)\|^2 -2(\Phi_{\xi\xi},u-\Phi),
\\ \|\Phi_\xi\|^2&=& \int_0^\infty \Big(V'(\xi-x)-B'(\xi)e^{-\lambda x}\Big)^2 dx
 = 2 \int_0^\theta \sqrt{F(s)}ds +O(1)e^{-2\lambda\xi}(1+\xi).
\eess

\begin{lem}\label{lemma 6.4}
There exist two positive constants $\xi_0$ and $\delta_0$ such that if $ d(u,\cM)\leqslant \delta_0 < \inf_{\xi \leqslant \xi_0 } d(\xi,u)$,
then there exists a unique $\xi >\xi_0 $ such that
\bess d(\xi,u) = d(u,\cM), \qquad d_\xi(\xi,u)=0, \qquad d_{\xi\xi}(\xi,u)>0. \eess
\end{lem}
\begin{proof}
Let $\xi_0$ be large enough so that, for any $\xi_1 > \xi_0$ and $\xi_1 < \xi \leqslant  \xi_1 +1$,
$$\frac{d}{d \xi} \| \Phi (\cdot ,\xi) - \Phi (\cdot , \xi_1) \| >0.$$
Such a $\xi_0$ exists because $|B (\xi)| + | B' (\xi)| \to 0$ as $\xi \to +\infty$.

Note also that $\inf_{\xi \geqslant 0} \| \Phi_\xi \|^2 >0$, so there exists $\delta_0$ such that,
for any $u \in L^2 ([0,\infty))$ and $\xi \geqslant 0$, $d(\xi,u) \leqslant 3\delta_0$ implies
$\frac{d^2}{d \xi^2} d(\xi,u) >0$. Moreover,  we choose $\delta_0$ is so small that
$\| \Phi (\cdot,\xi_1) - \Phi (\cdot,\xi)\|^2 \leqslant 2\delta_0$ implies $|\xi -\xi_1|<1$.
Now, given $u \in L^2 ([0,\infty))$ such that $d (u, \cM) \leqslant \delta_0 < \inf_{\xi \leqslant \xi_0 } d(\xi,u)$, proceed by contradiction and assume that there exists $\xi_0 < \xi_1 < \xi_2$ with $$d(u,\cM)=d( \xi_1,u)=d(\xi_2,u).$$
Then $\| \Phi (\cdot,\xi_1) - \Phi (\cdot,\xi_2)\|^2 \leqslant 2\delta_0$, and, from our choice of
$\delta_0$ and $\xi_0$, we have $\xi_2 < \xi_1+1$ and $\| \Phi (\cdot,\xi_1) - \Phi (\cdot,\xi)\|^2 \leqslant 2\delta_0$ for all $\xi \in (\xi_1 , \xi_2]$. It follows that $d(\xi,u) \leqslant 3\delta_0$ for any $\xi \in [\xi_1 ,\xi_2]$,  and $d(\xi,u)$ is strictly convex in the same interval, which immediately gives a contradiction. The rest of the lemma easily follows.
\end{proof}

\subsubsection{The Slow Motion Along the Manifold}

Let $u$ be a solution of  \eqref{p} satisfying \eqref{convergence H2} with $\xi(t)\to\infty$ as $t\to\infty$. From the lemma above, we can decompose it for large $t$ as
 \bess   u(x,t)= \Phi(x, y(t)) + v(x,t),\eess
 where $y(t)$ is the unique point such that
\bes\label{def-y(t)}
   \|u(\cdot,t)-\Phi(\cdot,y(t))\|^2= d(u(\cdot,t), \cM).
\ees
Then  $\|v\|\to 0$ and $y(t)-\xi(t)\to 0$ as $t\to \infty$. In particular, \eqref{convergence H2} also holds with $\xi (t)$ replaced by $y(t)$, so that the proof of Theorem~\ref{thm:main3} reduces to the study of the asymptotic behavior of~$y(t)$.

In the sequel, $\Phi(\cdot,y(t))$ is simply written as $\Phi$. Then by Lemma \ref{lemma 6.4}, we have
 $$v=u-\Phi,\quad  (v,\Phi_\xi)=0,\quad \|\Phi_\xi\|^2-(v,\Phi_{\xi\xi})>0. $$
In addition, the differential equation in \eqref{p} can be written as
\bes\label{basiceq}
    \dot y\; \Phi_\xi + v_t =R+ \cL^y v +N(\Phi,v)
\ees
where
\bess
   R=\Phi_{xx} +f(\Phi),\qquad \cL^y v=v_{xx}+f'(\Phi)v, \qquad N(\Phi,v)=  f(\Phi+v) - f(\Phi) - f'(\Phi)v .
\eess
Taking the inner product of (\ref{basiceq}) with $\Phi_{\xi}(\cdot,y(t))$ and using
$$(v_t,\Phi_\xi)=(v,\Phi_\xi)_t-(v,\Phi_{\xi\xi}) \dot y=-(v,\Phi_{\xi\xi})\dot y ,$$
we obtain
 \bess \Big(\| \Phi_\xi\|^2-(v,\Phi_{\xi\xi})\Big) \dot y = ( R,\Phi_\xi)
 +( \cL^y v,\Phi_\xi)+( N,\Phi_\xi).
 \eess
 Also, using the boundary condition  $\Phi-b \Phi_x|_{x=0}=0$  and $u-bu_x|_{x=0}=0$ we obtain
  $v-b v_x|_{x=0}=0$  and
 \bess
   ( \cL^y v,\Phi_\xi ) =   ( v,\cL^y \Phi_\xi) = ( v, R_\xi)  =O(1)\|v\|\;\|R_\xi\|=O(1)\sqrt{1+y} e^{-2\lambda y} \|v\|.
\eess
Similarly, using $N(\Phi,v)=O(1) v^2$ we obtain
 $  (N,\Phi_\xi) = O(1) \| v\|^2.$
   Hence, we have the motion law
\bess
  \dot y(t) &=& \frac{ ( R,\Phi_\xi)+(\cL^y v,\Phi_\xi)+(N,\Phi_\xi) }{\|\Phi_\xi\|^2- (v,\Phi_{\xi\xi})}
   \\ &=& \frac{(R,\Phi_\xi )} {\|\Phi_\xi\|^2+O(1)\|v\|} +
   O(1)   \sqrt{1+y}  e^{-2\lambda y}\|v\|+O(1)\|v\|^2
\\    &=& c(b) e^{-2\lambda y}+O(e^{-3\lambda y}) +  O(1)  \sqrt{1+y}  e^{-2 \lambda y}\|v\|+ O(1)\|v\|^2
 \eess
where $c(b)$ is that defined in Theorem~\ref{thm:main3}.
Hence, we have the following:

\begin{lem}\label{le5.5}   Under the decomposition $u(\cdot,t)=\Phi(\cdot,y(t))+v(\cdot,t)$ where
   $v\perp\Phi_\xi(\cdot,y(t))$, we have
   \bes\label{motion} \frac{dy(t)}{dt}  =  e^{-2\lambda y} c(b)+O(1) e^{-3\lambda y}+O(1)\|v\|^2.\ees
  When $b\lambda=1$  we have $c(b)=0$ and the next order expansion
  \bes  \frac{dy(t)}{dt}  = \hat c e^{-3\lambda y}+O(1) e^{-4\lambda y}+O(1)\|v\|^2,\quad \hat c=\frac{f''(0)A^3}{12 \int_0^\theta \sqrt{F(s)}ds}.\ees
\end{lem}

   \subsection{The Distance to the Approximated Center Manifold}
   We investigate the size $\|v\|$, starting   from the basic estimate
    \bess  \dot y = O(1) [e^{-2\lambda y} + \|v\|^2]. \eess
    We divide the estimate process   in several steps.

{\bf 1.} Let $(\nu_1(\xi),\varphi_1(\cdot,\xi))$ be the principal eigenpair of $\cL^\xi$.
Set
\bess c_1(t)=(v,\varphi_1(\cdot,y(t))),\quad  v_1= c_1(t) \varphi_1(\cdot,y(t)),\quad  v_2= v-v_1.\eess
Here we point out that $\varphi_1(x,y(t))\approx \varphi_1^0(x-y(t))$ decays exponentially fast as $|x-y(t)|\to \infty$,
so there exists a positive constant $C$ such that  (for all $t\gg1$)
$$\frac{1}{C} \leqslant\|\varphi_1 (\cdot, y(t))\|_{L^p([0,\infty))} \leqslant C \quad\forall\, p\in[1,\infty)\cup\{\infty\}.$$
Consequently, $\max\{\|v_1\|_{L^1}, \|v_1\|_{\infty}\} =O(1)|c_1(t)|=O(1) \|v_1\|$.
Also
 $v_2\perp \varphi_1$ and $\|v\|^2 = \|v_1\|^2 +\|v_2\|^2$. In addition,
\bess (v_{2t},v_{1}) &=& (v_2,v_1)_t-(v_2,v_{1t})=-(v_2,v_{1t})
\\ &=& -(v_2, \dot c_1 \varphi_1+c_1\varphi_{1\xi}\dot y)= -(v_2,\varphi_{1\xi}) \dot y
c_1(t) = O(1) \|v_2\|\;|\dot y|\;  \|v_1\|.
\eess
Hence,  taking the inner product of (\ref{basiceq}) with $v_1$ and using $v=v_1+v_2$ we obtain
 \bess  \frac{d}{dt}\frac{\|v_1\|^2}2 &=& -(v_{2t},v_1)+ (\cL^y v,v_1) +(R+N- \dot y \Phi_\xi,v_1)
\\ &=&O(1) \|v_2\| \;|\dot y|\;\|v_1\|+ ( v,\cL^y v_1) + O(1) \|v_1\| (e^{-2\lambda y}+\|v\|^2).
\\ &=& \nu_1 \|v_1\|^2 + O(1) [  e^{-2\lambda y}+\|v\|^2 ]\; \|v_1\|.\eess
Here we have used the estimates
\bess
& |(R-\dot y\Phi_\xi,v_1)| \leqslant \|R-\dot y\Phi_\xi\|_{\infty}\; \|v_1\|_{L^1} = O(1) (  e^{-2\lambda y}+\|v\|^2) \|v_1\|,\\
& |(N,v_1)|\leqslant   \|N\|_{L^1} \|v_1\|_{\infty}= O(1) \| v\|^2\|v_1\|. \eess
Hence
\bess
\frac{d}{dt} \|v_1\| &=& \nu_1 \|v_1\| + O(1) [ e^{-2\lambda y(t)}+\|v\|^2].
\eess

{\bf 2.}
 Similarly,  taking the inner product of (\ref{basiceq}) with $v_2$   we obtain
\bess
\frac{d}{dt}\frac{\|v_2\|^2}{2} &=& - ( v_{1t},v_2)+(\cL^y v,  v_2 ) + (R+N-\dot y\Phi_\xi,v_2)
\\ &\leqslant & \nu_2 \|v_2\|^2 + O(1) \Big[ \sqrt{1+y} e^{-2\lambda y}+\|v\|_\infty \|v\| + \|v\|^2 \Big] \|v_2\| .
\eess
Here we used the fact that $v_2\perp \varphi_1$ so $(\cL^y v,v_2)=(\cL^y v_1+\cL^y v_2,v_2)=
(\cL^y v_2,v_2) \leqslant \nu_2 \|v_2\|^2$, and the estimate
\bess |(R,v_2)|\leqslant \|R\|\;\|v_2\| = O(1)  \sqrt{1+y}  e^{-2\lambda y} \|v_2\|,
\quad (N,v_2) = O(1) \|v\|_{\infty} \|v\|\;\|v_2\|. \eess
Thus,
\bess
   \frac{d}{dt} \|v_2\| \leqslant \nu_2 \|v_2\|
    + O(1) [\sqrt{1+y}e^{-2\lambda y(t)}+\|v\|_\infty\|v\|  +\|v\|^2].
\eess

{\bf 3.}
Combining  the above two estimates, we then obtain, for $M>0$ and $y>1/\lambda$,
\bess && \frac{d}{dt}\Big( \|v_1\|-\|v_2\|- M\sqrt{1+y} e^{-2\lambda y}\Big) \geqslant \frac{d}{dt}\Big( \|v_1\|-\|v_2\|\Big) -2\lambda M\sqrt{1+y} e^{-2\lambda y}|\dot y| \vspace{3pt}
 \\  &\geqslant&  \nu_1 \|v_1\|-\nu_2\|v_2\|-O(1)
  [ \sqrt{1+y} e^{-2\lambda y}+
 \|v\|_\infty \|v\|+\|v\|^2]-2\lambda M \sqrt{1+y} e^{-2\lambda y}|\dot y| \vspace{3pt}
\\ &\geqslant & \frac{\nu_1}{2} (\| v_1\|-\|v_2\|- M\sqrt{1+y} e^{-2\lambda y}) + \|v_1\|\Big( \frac{\nu_1}{2} -O(1) [ \|v\| +\|v\|_\infty] \Big) \vspace{3pt}
\\ && \quad + \|v_2\|\Big( \frac{\nu_1}{2} -\nu_2 -O(1) [ \|v\| +\|v\|_\infty]\Big)  +  \sqrt{1+y} e^{-2\lambda y} \Big( \frac{\nu_1}2 M -2\lambda M|\dot y|-O(1)\Big).
 \eess
The lemma below follows:
 \begin{lem}
 Under the notation above, there exist positive constants $M$ and $t_0>0$ such that
 \bess \|v_1\| &\leqslant & M\sqrt{1+y} e^{-2\lambda y}+\|v_2\| \qquad\forall\, t>t_0. \eess
 \end{lem}

\begin{proof} Since $\lim_{\xi\to\infty} \nu_1(\xi)=\mu_1>0$ and   $\lim_{\xi\to\infty}\nu_2(\xi)=0$, we can choose $M>0$ and $t_0 >0$ such that for all $t >t_0$:
\bess  \frac{\nu_1}{2} -O(1) [ \|v\| +\|v\|_\infty] >0,
\\ \frac{\nu_1}{2} -\nu_2 -O(1) [ \|v\| +\|v\|_\infty] >0,
\\ \frac{\nu_1}2M -2\lambda M |\dot y(t)|-O(1)\geqslant 0.
\eess
This implies that
\bess \frac{d}{dt} \Big( \|v_1\|-\|v_2\|-M\sqrt{1+y} e^{-2\lambda y}\Big) \geqslant
\frac{\nu_1}{2} \Big( \|v_1\|-\|v_2\|-M\sqrt{1+y} e^{-2\lambda y}\Big).\eess
Thus, we must have the assertion since otherwise $\|v_1\|-\|v_2\|-M\sqrt{1+y}
e^{-2\lambda y}$ grows exponentially fast.
\end{proof}

{\bf 4.}
Next, we estimate $v_2$. Let $(\nu_2(y),\varphi_2(\cdot,y))$ be the second eigenpair of $\cL^y$.  We set
\bess c_2= (v_2, \varphi_2), \qquad \hat v_{2} = c_2 \varphi_2, \quad v_3 = v_2-\hat v_2. \eess
Then $v=v_1+\hat v_2+v_3$. Decomposing $\Phi_\xi= a_1 \varphi_1+a_2 \varphi_2+\Phi_\xi^\perp$
where $\Phi_\xi^\perp\perp \varphi_1,\varphi_2$. Since $v\perp \Phi_\xi$ we obtain
\bess  0=(v,\Phi_\xi)= (c_1\varphi_1+c_2\varphi_2 +v_3,a_1 \varphi_1+a_2\varphi_2+\Phi_\xi^\perp)
=a_1 c_1+a_2 c_2 + (v_3,\Phi_\xi^\perp).  \eess
Hence, we have
\bess c_2 = -\frac{1}{a_2} \Big( a_1 c_1+ (v_3,\Phi_\xi^\perp)\Big).\eess
Hence, when $t>t_0$,
\bess \|v\|& = & \|v_1+v_2\|\leqslant \|v_1\|+\|v_2\|
\\ &\leqslant &  M\sqrt{1+y} e^{-2\lambda y} + 2 \| v_2\|
\leqslant M\sqrt{1+y} e^{-2\lambda y}+2|c_2|+ 2\|v_3\|
\\ &\leqslant & M\sqrt{1+y} e^{-2\lambda y}+ \frac{2}{|a_2|} \Big(|a_1| \|v_1\|+ \|v_3\|\;\|\Phi_\xi^\perp\|\Big)+2 \|v_3\|.
\eess
Since $\|v\|^2=\|v_1\|^2+\|\hat{v}_2\|^2+\|v_3\|^2$ and
\bess
   \lim_{t\to\infty} (|a_1|+\|\Phi_\xi^\perp\|)=0,\qquad \lim_{t\to\infty} a_2=\|V'\|_{L^2(\R)} =\Big( 2\int_0^{\theta}
   \sqrt{F(s)} ds \Big)^{1/2},
\eess
we see that there exist constants $t_1>0$ and $\gamma>0$ such that
\bes\label{v-v_3}
   \|v\| \leqslant  2M\sqrt{1+y} e^{-2\lambda y}  + \gamma \|v_3\|\quad\forall\, t>t_1.
\ees

{\bf 5.} Finally, writing $v=c_1\varphi_1+c_2\varphi_2+v_3$, and taking the inner product of
\eqref{basiceq} with $v_3$ we obtain
\bess \frac{d}{dt}\frac{\|v_3\|^2}{2} &=& (-[c_1 \varphi_{1}+c_2\varphi_{2}]_t, v_3) + (\cL^y v, v_3)
+ (R-\dot y \Phi_\xi+N, v_3)
\\ &=& (-c_1 \varphi_{1\xi} -c_2\varphi_{2\xi}, v_3) \dot y + (\cL^y v_3,v_3) + (R-\dot y \Phi_\xi+N, v_3)
\\ &\leqslant & \nu_3 \|v_3\|^2 + O(1) [\sqrt{1+y} e^{-2\lambda y}+\|v\|_\infty \|v\| +\|v\|^2] \|v_3\|.
\eess
Here in the second equation we have used the fact that $(\dot c_i \varphi_i,v_3)=0$ for $i=1,2$.
Thus,
\bess
\frac{d}{dt} \|v_3\| &\leqslant&  \nu_3\|v_3\| +  O(1) [\sqrt{1+y} e^{-2\lambda y}+\|v\|_\infty\|v_3\| +\|v\|^2 ]
\\ &\leqslant& \Big(\nu_3+O(1)[\|v\|_\infty+\|v\|] \Big) \|v_3\|+O(1) \sqrt{1+y} e^{-2\lambda y}
\eess
by \eqref{v-v_3}. Hence,  when $y>1/\lambda$,
 \bess \frac {d}{dt}\Big( \|v_3\|- \hat M\sqrt{1+y}e^{-2\lambda y}\Big)
  & \leqslant & \|v_3\|\Big(\nu_3 +  O(1) [\|v\|_\infty+\|v\|] \Big)+ [2\lambda\hat M |\dot y|+O(1)] \sqrt{1+y}e^{-2\lambda y}
   \\ &=&-\nu \Big(\|v_3\|-\hat M\sqrt{1+y} e^{-2\lambda y}\Big)
  \\ &&
   + \|v_3\|\Big({\nu_3}+\nu + O(1)[\|v\|_\infty + \|v\|]\Big)
  \\ &&  +
     \sqrt{1+y} e^{-2\lambda y}\Big(-\nu\hat M+2\lambda|\dot y|\hat M +O(1)\Big)
 \eess
 where $\nu$ and $\hat M$ are positive constants to be determined.

Since $\lim_{t\to\infty}\nu_3(y(t))=\mu_3<0$, we can choose $\nu\in(0,-\mu_3)$, $t_2$ and $\hat M\gg1$ such that when $t>t_2$,
\bess
\frac{d}{dt} \Big(\|v_3\|-\hat M\sqrt{1+y} e^{-2\lambda y}\Big) \leqslant -\nu
  \Big( \|v_3\|-\hat M\sqrt{1+y} e^{-2\lambda y}\Big).
\eess
This implies that, for some $C>0$,
$$
\|v_3\|-\hat M\sqrt{1+y} e^{-2\lambda y} \leqslant C e^{-\nu (t-t_2) }\quad\forall\, t>t_2.
$$
Putting this together with \eqref{v-v_3}, we conclude that
\begin{equation}\label{eqn:expansion_v}
 \|v(\cdot,t)\| = O(1)[ \sqrt{1+y(t)} e^{-2\lambda y(t)} + e^{-\nu t}].
\end{equation}

\subsection{Completion of the Proof of Theorem~\ref{thm:main3}}

Substituting \eqref{eqn:expansion_v} into the assertions of Lemma \ref{le5.5} we obtain the following:

\begin{lem}\label{lem:speed}
  Assume that \eqref{convergence H2} holds with $\xi(t)\to \infty$ as $t\to \infty$. Let $y(t)$
  be the function defined by \eqref{def-y(t)}. Then
  \bes   \label{doty} \dot y(t) &=&  c(b) e^{-2\lambda y}+ O(e^{-3\lambda y}) + O(1) e^{-2\nu t}. \ees
 When $b\lambda =1$  we have $c(b)=0$ and the next order expansion
\bes\label{doty2}    \dot y(t) &=&  \hat c  e^{-3\lambda y}+ O(e^{-4\lambda y}) + O(1) e^{-2\nu t}.\ees
\end{lem}

  Using this lemma we can prove the following results, from which Theorem~\ref{thm:main3} is an immediate consequence.

  \begin{thm}\label{thm:speed}
  Assume that \eqref{convergence H2} holds with $\xi(t)\to \infty$ as $t\to \infty$. Let $y(t)$
  be the function defined by \eqref{def-y(t)}. Then $y(t)-\xi(t)\to 0$ as $t\to \infty$ and the following holds:
  \begin{enumerate}
 \item[\rm (i)] when $0\leqslant b\lambda <1$,
\bess
  y(t)  = \frac{1}{2\lambda} \ln [2\lambda c(b) t]+ \frac{O(1)}{\sqrt{t}},
  \quad \dot y(t)=\frac{1}{2\lambda t}+\frac{O(1)}{t^{3/2}};
\eess

 \item[\rm (ii)] when $b\lambda=1$, $f\in C^3$ and $f''(0)>0$,
\bess  y(t)  = \frac{1}{3\lambda} \ln [3\lambda \hat c t]+ \frac{O(1)}{t^{1/3}},
 \quad \dot y(t)=\frac{1}{3\lambda t}+\frac{O(1)}{t^{4/3}}.\eess


  \end{enumerate}
  \end{thm}

\begin{proof}
We already mentioned that $y(t)-\xi(t)\to 0$ as $t\to \infty$, which immediately follows from the conclusions in the previous section.
Since $\lim_{t\to\infty} y(t)=\infty$,  we obtain  from \eqref{doty}  that
$\dot y(t)=o(1)$ so $y(t)=o(t)$. This implies that $e^{-2\nu t}=O(1) e^{-4\lambda y}$.
 Equation \eqref{doty} can be rewritten as
\bess  e^{2\lambda y} \frac{dy}{dt} = c(b)+O(1) e^{-\lambda y}. \eess

(i) If $0\leqslant b\lambda<1$, then $c(b)>0$.  An  integration gives
   \bess  y(t) &=& \frac{1}{2\lambda} \ln\Big\{e^{2\lambda y(0)} +2\lambda c(b) t +O(1)\int_0^t  e^{-\lambda y(\tau)}d\tau \Big\}
   \\ &=& \frac{1}{2\lambda} \ln [2\lambda c(b) t] + O(1) \frac{ e^{2\lambda y(0)}+
   \int_0^t e^{-\lambda y(\tau)}d\tau}{t}.\eess
   Hence, $y(t)=\frac{1}{2\lambda} \ln[2\lambda c(b) t]+O(1)$,  $ \int_0^t e^{-\lambda y(\tau)}d\tau=O(\sqrt{t})$, $y(t)=\frac{1}{2\lambda}\ln[2\lambda c(b) t]+O(1)t^{-1/2}$, and by  \eqref{doty}, $\dot y(t)=\frac{1}{2\lambda t}+\frac{O(1)}{t^{3/2}}$.

 (ii) The case $b\lambda=1$ and $f''(0)>0$ is similarly treated as (i),  using (\ref{doty2}) instead of (\ref{doty}).
\end{proof}


\bigskip

{\bf Acknowledgement}.  The second author would like to thank Professors E. Feireisl,
P. Pol\'{a}\v{c}ik, Y. Yamada, Y. Du for sending their papers to them and/or helpful discussions.

\end{document}